\documentclass{amsart}

\usepackage[a4paper,outer=3.5cm,inner=3.5cm,total={14cm,21.4cm}]{geometry} 
\usepackage[english]{babel}
\usepackage[utf8]{inputenc}
\usepackage{amsmath}
\usepackage{amssymb}
\usepackage{amsthm}
\usepackage{enumerate}
\usepackage[shortlabels]{enumitem}
\usepackage{tikz-cd}
\usepackage{stmaryrd}
\usepackage{verbatim}
\usepackage{framed}

\setlist[enumerate]{label=\arabic*.}

\usepackage[colorlinks=true]{hyperref}
\hypersetup{colorlinks, citecolor=blue, filecolor=black, linkcolor=black, urlcolor=black}

\usepackage[capitalize]{cleveref}
\crefformat{equation}{\normalfont(#2#1#3)}

\newtheorem{thm}{Theorem}[section]
\newtheorem{Theorem}[thm]{Theorem}
\newtheorem{Lemma}[thm]{Lemma}
\newtheorem{Proposition}[thm]{Proposition}
\newtheorem{Corollary}[thm]{Corollary}
\newtheorem{claim}[thm]{Claim}

\theoremstyle{definition}
\newtheorem{Definition}[thm]{Definition}

\newtheorem{Remark}[thm]{Remark}

\newtheorem{step}{Step}

\newcommand{\B}{\mathbb{B}}

\newcommand{\F}{\mathbb{F}}
\newcommand{\G}{\mathbb{G}}
\newcommand{\N}{\mathbb{N}}
\renewcommand{\O}{\mathcal{O}}

\newcommand{\Q}{\mathbb{Q}}

\newcommand{\Z}{\mathbb{Z}}
\newcommand{\m}{\mathfrak{m}}

\newcommand{\coker}{\operatorname{coker}}

\newcommand{\Hom}{\operatorname{Hom}}
\newcommand{\Map}{\operatorname{Map}}

\newcommand{\HOM}{\underline{\mathrm{Hom}}}

\newcommand{\Perf}{\operatorname{Perf}}

\newcommand{\Spa}{\operatorname{Spa}}

\newcommand{\cts}{{\operatorname{cts}}}
\newcommand{\lc}{{\operatorname{lc}}}
\newcommand{\an}{\mathrm{an}}
\newcommand{\hotimes}{\hat{\otimes}}

\newcommand{\et}{{\mathrm{\acute{e}t}}}

\newcommand{\etqcqs}{{\mathrm{\acute{e}t}\text{-}\mathrm{qcqs}}}
\newcommand{\stdet}{\mathrm{std\acute{e}t}}

\newcommand{\HT}{\operatorname{HT}}
\newcommand{\aeq}{\stackrel{a}{=}}

\newcommand{\wh}{\widehat}
\newcommand{\bOx}{\overline{\O}^{\times}}

\newcommand{\Pic}{\operatorname{Pic}}
\newcommand{\uP}{\uPic}
\newcommand{\NS}{\mathrm{NS}}
\newcommand{\LSD}{\mathrm{LSD}}
\newcommand{\Spd}{\mathrm{Spd}}
\newcommand{\wt}{\widetilde}
\renewcommand{\diamond}{\diamondsuit}

\newcommand{\twolim}{2\text{-}\varinjlim}

\newcommand{\Char}{\operatorname{char}}
\renewcommand{\lim}{\varprojlim}

\newcommand{\cH}{{\ifmmode \check{H}\else{\v{C}ech}\fi}}
\newcommand{\tf}{[\tfrac{1}{p}]}
\renewcommand{\tt}{\mathrm{tt}}

\newcommand{\uPic}{\mathbf{Pic}}
\newcommand{\wtOm}{\wt\Omega}
\newcommand{\Oone}{\O^\times_1}

\newcommand*\isomarrow{%
	\xrightarrow{\raisebox{-0.35em}{\smash{\ensuremath{\sim}}}}
}  

\tikzset{
	labelrotate/.style={anchor=south, rotate=90, inner sep=.5mm}} 
\tikzset{
	labelrotatep/.style={anchor=north, rotate=90, inner sep=.75mm}}

\title{Diamantine Picard functors of rigid spaces}
\author{Ben Heuer}
\date{}

\usepackage[foot]{amsaddr}
\address{University of Frankfurt,
	Robert-Mayer-Str. 6-8,
	60325 Frankfurt am Main, Germany}
\email{heuer@math.uni-frankfurt.de}
\makeatletter
\@namedef{subjclassname@2020}{2020 Mathematics Subject Classification}
\makeatother
\begin{document}
\maketitle\thispagestyle{empty}
\begin{abstract}
	For a connected smooth proper rigid space $X$ over a perfectoid field extension of $\Q_p$, we show that the \'etale Picard functor of $X^\diamond$ defined on perfectoid test objects is the diamondification of the rigid analytic Picard functor. In particular, it is represented by a rigid analytic group variety if and only if the rigid analytic Picard functor is. 
	
	Second, we study the $v$-Picard functor  that parametrises line bundles in the finer $v$-topology on the diamond associated to $X$ and relate this to the rigid analytic Picard functor by a geometrisation of the multiplicative Hodge--Tate sequence.
	
	The motivation is an application to the $p$-adic Simpson correspondence, namely our results pave the way towards the first instance of a new moduli theoretic perspective.
\end{abstract}

\setcounter{tocdepth}{2}
\section{Introduction}
Line bundles are ubiquitous in rigid analytic geometry, and
Picard groups of rigid spaces have therefore been the subject of extensive studies, for example in \cite{Gerritzen-ZerlegungPicard,BL-stable-reduction-II,HartlLutk,Conrad_Ampleness,Lutkebohmert_RigidCurves,Kerz_towards_rigid_BQ,li2017rigid}. It is natural to ask how much of this theory carries over  to Scholze's larger categories like the pro-\'etale site \cite{Scholze_p-adicHodgeForRigid} or the category of diamonds \cite{etale-cohomology-of-diamonds}.
One particular question with a long history in rigid geometry is about the existence of rigid analytic moduli spaces of line bundles on proper rigid spaces, namely about representability of the rigid Picard functor which parametrises isomorphism classes of line bundles: 

 Let $p$ be a prime, let $K$ be a perfectoid extension of $\Q_p$ and let 
$\pi:X\to \Spa(K)$ be a smooth proper rigid space considered as an adic space. 
The rigid Picard functor is  the sheaf
\[\uPic_{X,\et}:=R^1\pi_{\et\ast}\G_m:\mathrm{SmRig}_{K,\et}\to \mathrm{Ab}\]
where $\mathrm{SmRig}_{K,\et}$ is the big \'etale site of smooth rigid spaces over $K$ and $\pi_{\et\ast}$ is the pushforward of big \'etale sites along $\pi$. 
Explicitly, $\uPic_{X,\et}$ is the sheafification of the functor sending a smooth rigid space $Y$ to the group of isomorphism classes of line bundles on $X\times Y$. It is expected that $\uPic_{X,\et}$ is always representable, and this is known in many cases of interest:

\begin{enumerate}
\item If $X$ is the analytification of a smooth proper algebraic variety $X_0$, then it follows from K\"opf's relative rigid GAGA-Theorem \cite{Koepf_GAGA}\cite[Theorem~2.8]{Lutkebohmer-FARAG}\cite[Example 3.2.6]{Conrad_Ampleness} that $\uPic_{X,\et}$ is the analytification of the algebraic Picard variety of $X_0$ \cite[\S1]{BL-stable-reduction-II}. In particular, the identity component $\uPic^0_{X,\et}$ is then an abelian variety.  
\item In \cite[\S6]{BL-Degenerating-AV}, Bosch--L\"utkebohmert treated the case that $X$ is an abeloid variety, i.e.\ a connected smooth proper rigid group variety: $\uPic^0_{X,\et}$ is then the dual abeloid variety.
\item Hartl--L\"utkebohmert \cite{HartlLutk} proved that if $X$ has a strict semi-stable formal model over a discrete valuation ring, then  $\uPic_{X,\et}$ is represented by a rigid group such that  $\uPic^0_{X,\et}$ is semi-abeloid, i.e.\ an extension of an abeloid variety by a torus.
\item Warner has announced in his thesis \cite{warner2017adic} a proof that $\uPic_{X,\et}$ is always representable, but he does not describe what $\uPic^0_{X,\et}$ looks like in general.
\end{enumerate}

\subsection{The \'etale diamantine Picard functor}

The goal of this article is to study Picard functors that are instead defined on perfectoid spaces as defined by Scholze \cite{perfectoid-spaces}. Viewing rigid spaces through perfectoid spaces naturally leads us into the setting of diamonds introduced in \cite{etale-cohomology-of-diamonds}:
Let $\Perf_{K,\et}$ be the site of perfectoid spaces over $X$ equipped with the \'etale topology, and $\Perf_{K,v}$ the same category with the much finer $v$-topology from \cite[\S8]{etale-cohomology-of-diamonds}.
Associated to $X$ we have the diamondification \[\pi^{\diamondsuit}:X^{\diamondsuit}\to \mathrm{Spd}(K)\]
defined in \cite[\S15]{etale-cohomology-of-diamonds}. This is a morphism of diamonds, and thus of sheaves on $\Perf_{K,v}$. Analogously to the rigid case, we define the \'etale diamantine Picard functor to be the sheaf
\[\uPic^{\diamondsuit}_{X,\et}:=R^1\pi^\diamond_{\et\ast}\G_m:\Perf_{K,\et}\to \mathrm{Ab}.\]
Explicitly, this is the \'etale sheafification of the functor that sends a perfectoid space $Y$ over $K$ to the set of isomorphism classes of line bundles on the analytic adic space $X\times Y$.

The first main result of this article is that this new diamantine Picard functor $\uPic^{\diamondsuit}_{X,\et}$ can be described in terms of the rigid analytic Picard functor $\uP_{X,\et}$. For this introduction, let us for simplicity assume that $\uP_{X,\et}$ is represented by a rigid space. Then we show:
\begin{Theorem}\label{t:representability-of-et-Picard-functor-intro}
	There is a natural isomorphism of sheaves on $\Perf_{K,\et}$
	\[(\uP_{X,\et})^\diamondsuit\isomarrow \uP^{\diamondsuit}_{X,\et},\]
	that is, $\uP^{\diamondsuit}_{X,\et}$ is represented by the diamondification of the rigid analytic Picard functor.
\end{Theorem}
\begin{Corollary}
	If $X$ is connected and $x\in X(K)$ is a point, then for any perfectoid space $T$ over $K$, the isomorphism classes of line bundles on $X\times T$ that are trivial over $x\times T$ are in natural one-to-one correspondence with the morphisms of adic spaces $T\to \uP_{X,\et}$ over $K$.
\end{Corollary}

In fact, we define more generally a natural ``diamondification of functors on $\mathrm{SmRig}_{K,\et}$'' such that the statement of \cref{t:representability-of-et-Picard-functor-intro} holds without requiring $\uP_{X,\et}$ to be representable: Explicitly, the result then says that for every perfectoid space $T$ and any line bundle $L$ on $X\times T$, one can \'etale-locally on $T$ find a rigid space $T_0$ and a morphism $T\to T_0$ such that $L$ descends to $X\times T_0$. This statement is not at all obvious, and it is false in general if we remove the assumption that $X$ is proper.

\subsection{The $v$-Picard functor}

In the diamantine setting, there is now also a second natural Picard functor 
\[\uPic^{\diamondsuit}_{X,v}:=R^1\pi^\diamond_{v\ast}\G_m:\Perf_{K,v}\to \mathrm{Ab},\]
given by the $v$-sheafification of the functor that sends $Y$ to the sheaf of isomorphism classes of $v$-line bundles on $X\times Y$. This difference in topology is more than just a technicality: We showed in  \cite[Theorem~1.3.2a]{heuer-v_lb_rigid} that there  are in general many more \mbox{$v$-line} bundles on $X$ than \'etale line bundles. In fact,  we showed that if $K$ is algebraically closed, the respective Picard groups of isomorphism classes fit into a ``multiplicative Hodge--Tate sequence'':
\[0\to \Pic_{\et}(X)\to \Pic_v(X)\to H^0(X,\wtOm^1)\to 0\]
where $\wtOm^1:=\Omega^1\{-1\}$ is a Tate twist.
Our second main result is that this short exact sequence can be geometrised to give a comparison between the \'etale and $v$-topological Picard functors:

\begin{Theorem}\label{t:representability-of-v-Picard-functor-intro}
		Let $X$ be a smooth proper rigid space over a perfectoid extension $K$ of $\Q_p$. Then
		the v-Picard functor fits into a natural exact sequence of abelian sheaves on $\Perf_{K,\et}$
		\[ 0\to \uP_{X,\et}^\diamondsuit\to \uP_{X,v}^\diamondsuit\to H^0(X,\wt\Omega^1_X)\otimes_K \G_a\to 0.\]
		In particular, $\uP_{X,v}^\diamondsuit$ is represented by a rigid group variety if and only if $\uP_{X,\et}$ is.
\end{Theorem}
	One  consequence is that  $\uP_{X,\et}^\diamondsuit$ is in fact a $v$-sheaf, so we can regard $\uP_{X,\et}^\diamondsuit$ as an extension of $\uP_{X,\et}$ to a much larger class of test objects.
	Conversely, we show that if $\uPic_{X,\et}^\diamondsuit$ is represented by a rigid group, then this also represents $\uPic_{X,\et}$. As we explain in more detail below, this might open up new methods to study $\uPic_{X,\et}$ and its representability.

\subsection{Applications to non-abelian $p$-adic Hodge theory} \cref{t:representability-of-et-Picard-functor-intro} and \cref{t:representability-of-v-Picard-functor-intro} are of interest in their own right, but
our main motivation for studying diamantine Picard varieties is an application to the $p$-adic Simpson correspondence: 

The reason why perfectoid test objects appear naturally in this context is that in \cite{heuer-geometric-Simpson-Pic} we describe a class of ``topological torsion'' line bundles $L$ on $X$ characterised by the property that $L$ extends to a line bundle on the adic space $X\times \wh{\Z}$ such that the specialisation of $L$ at $n\in \Z\subseteq \wh{\Z}$ is isomorphic to $L^n$. We would like these to be precisely those line bundles which induce a homomorphism of adic groups $\wh{\Z}\to \uPic_X$. This is guaranteed by \cref{t:representability-of-et-Picard-functor-intro}.

\medskip

While this was our original motivation to study diamantine Picard functors, \cref{t:representability-of-v-Picard-functor-intro} exhibits a much deeper connection, namely it paves the way for a new understanding of the $p$-adic Simpson correspondence as a geometric comparison of moduli spaces:
 The short exact sequence in \cref{t:representability-of-v-Picard-functor-intro} gives a geometrisation of the equivalence between $v$-line bundles and Higgs line bundles of \cite[\S5]{heuer-v_lb_rigid}. Geometrically speaking, it says that there is a rigid analytic moduli space of $v$-line bundles $\uPic_{X,v}^\diamondsuit$ which is a  $\uPic_{X,\et}^\diamondsuit$-torsor over $\mathcal A:=H^0(X,\wtOm^1_X)\otimes_K \G_a$ in a natural way. As the same is true for the moduli space $\mathbf{Higgs}_1:=\uPic_{X,\et}^\diamondsuit\times\mathcal A$ of Higgs line bundles, we find that $\uPic_{X,v}^\diamondsuit$ is an \'etale twist of $\mathbf{Higgs}_1$ over $\mathcal A$. As we explain in detail in the sequel \cite{heuer-geometric-Simpson-Pic},  this leads to a geometric non-abelian Hodge correspondence in rank one. 
 
 The present article is therefore the first in a series on the geometrisation of the p-adic Simpson correspondence. Indeed, in \cite{HX}, we will show that when $X$ is a curve, the perspective that the moduli space of v-vector bundles is a twist of the moduli space of Higgs bundles generalises to objects of higher rank.

\subsection{The topological torsion Picard functor is representable}\label{s:tt-Pic-intro}
Given the known results about representability of the rigid Picard functor which we summarised above, it seems plausible that $\uPic_{X,\et}$ is always representable by a rigid group variety whose identity component is a semi-abeloid variety.

In order to illustrate how the perfectoid perspective can help understand the structure of Picard functors, let us already mention the following result from the sequel article \cite{heuer-geometric-Simpson-Pic}, saying that a topological torsion version of the rigid Picard functor is always representable:
Let $\wh\G_m$ be the subgroup of topologically $p$-torsion units, given by the open disc at 1 of radius 1. We then define the topologically $p$-torsion Picard functor of $X$ to be
\[\widehat{\uPic}_{X,\et}:=R^1\pi_{\et\ast}\wh\G_m:\mathrm{SmRig}_{K,\et}\to \mathrm{Ab}\]
Using a geometric $p$-adic Simpson correspondence in terms of ${\uPic}^{\diamondsuit}_{X,\et}$, we prove in \cite{heuer-geometric-Simpson-Pic} that
	$\widehat{\uPic}_{X,\et}$ is always representable by a finite disjoint union of analytic \mbox{$p$-divisible} groups in the sense of Fargues \cite[\S 1.6]{Fargues-groupes-analytiques}. As such, it is a subgroup of $\HOM(\pi_1(X),\wh\G_m)$ where $\pi_1(X)$ is the \'etale fundamental group of $X$. If $\uPic_{X,\et}$ is represented by a rigid group $G$, then this is the topological $p$-torsion subgroup of $G$ as defined by Fargues \cite[\S 1.6]{Fargues-groupes-analytiques}.

This produces some evidence that $\uP_{X,\et}$ is always representable by a rigid group whose connected component is a semi-abeloid variety: Namely, it imposes restrictions on what kind of rigid groups can appear as Picard varieties, which are consistent with this prediction.

\subsection*{Acknowledgements}
We thank Johannes Ansch\"utz, Urs Hartl, Lucas Mann, Alexander Petrov, Peter Scholze, Annette Werner, Daxin Xu and Bogdan Zavyalov for very useful conversations. 

We thank Lucas Gerth for pointing out several mistakes in an earlier version. Special thanks to the referee for many helpful comments which have greatly improved the article.

This work was funded by the Deutsche Forschungsgemeinschaft (DFG, German Research Foundation) under Germany's Excellence Strategy-- EXC-2047/1 -- 390685813. The author was supported by the DFG via the Leibniz-Preis of Peter Scholze.
\subsection*{Notation}
Let $K$ a complete non-archimedean field of residue characteristic $p$. Let $\O_K$ be the ring of integers of $K$ and let $K^+\subseteq \O_K$ be any ring of integral elements.
Let $\m\subseteq K^+$ be the ideal of topologically nilpotent elements and fix a pseudo-uniformiser $\varpi\in \m$.  We use adic spaces in the sense of Huber \cite{Huber-ageneralisation}. We abbreviate $\Spa(K,K^+)$ by $\Spa(K)$ when $K^+$ is clear from the context.

In this article, a rigid space over $K$ is by definition an adic space locally of topologically finite type over $\Spa(K,K^+)$.  In particular, we then denote by $\mathrm{SmRig}_{K}$ the category of smooth rigid spaces over $\Spa(K,K^+)$.  For any rigid space $X$ we denote by $\mathrm{SmRig}_{X}$ the slice category of rigid spaces over $X$. For rigid spaces $X$, $Y$ over $K$, we also write $Y_X:=Y\times_KX$. 
We denote by $\B^d$ the unit disc of dimension $d$ over $K$, considered as an affinoid rigid space.

 We are most interested in the case of $K^+=\O_K$, in which case the notion of rigid spaces is equivalent to the classical one under mild technical assumptions, but like in \cite{Scholze_p-adicHodgeForRigid} it is useful to consider the more general case since this makes it very easy to later pass to the relative situation of morphisms of rigid spaces $\pi:X\to Y$, which is useful for applications.

 We use perfectoid spaces in the sense of \cite{perfectoid-spaces} and denote by $\Perf_{K}$ the category of perfectoid spaces over $(K,K^+)$. We use diamonds in the sense of \cite{etale-cohomology-of-diamonds} and denote by $\mathrm{LSD}_{K}$ the category of locally spatial diamonds over $\Spd(K,K^+)$. For any $X\in \mathrm{LSD}_{K}$, we denote by $\mathrm{LSD}_{X}$ the slice category.

By a rigid group over $K$ we mean a group object in the category of rigid spaces over $K$, always assumed to be commutative. We denote by $\G_m$ the rigid multiplicative group, by $\G_a$ the rigid additive group, and by $\G_a^+\subseteq \G_a$ the subgroup given by the closed unit ball $\B^1$.

\section{The diamantine Picard functors}
Let $K$ be a perfectoid field over $\Q_p$ and let $\pi:X\to \Spa(K,K^+)$ be a smooth rigid space. 
The aim of this section is to introduce the diamantine Picard functors of $X$, and to state a more precise  and more general version of the main results, as well as some corollaries. We do not give any proofs yet, but we end the section with an overview of the strategy of proof.

\subsection{Definition of diamantine and rigid Picard functors}
Recall that the rigid analytic Picard functor as considered in \cite{HartlLutk} can be defined as the abelian sheaf
\[\uPic_{X,\et}:=R^1\pi_{\et\ast}\G_m:\mathrm{SmRig}_{K,\et}\to \mathrm{Ab}\]
where $\mathrm{SmRig}_{K,\et}$ is the site of smooth rigid spaces over $\Spa(K,K^+)$ with the \'etale topology and
where $\pi_{\et}:\mathrm{SmRig}_{X,\et}\to \mathrm{SmRig}_{K,\et}$ is the natural morphism of big \'etale sites. Conjecturally, if $X$ is proper, then $\uPic_{X,\et}$ is represented by a smooth rigid group variety. As summarised in the introduction, this is known in many cases, but not yet in full generality.

\medskip 

Our goal in this subsection is to introduce a ``diamantine'' variant of the Picard functor defined on perfectoid test objects, and to explain how this can be compared to $\uP_{X,\et}$.

Recall from \cite[\S15]{etale-cohomology-of-diamonds} that there is a fully faithful diamondification functor
\[ \mathrm{SmRig}_K\to\LSD_K ,\quad X\mapsto X^\diamond\]
sending a smooth rigid space $X$ to its associated locally spatial diamond over $\Spd(K,K^+)$. We sometimes drop $-^\diamondsuit$ from notation when this is clear from the context, for example we simply write $\G_m$ for the diamond that sends a perfectoid space $\Spa(R,R^+)$ to $R^\times$.
We write $\delta$ for the morphism of \'etale sites associated to the above functor (cf \cite[Lemma~15.6]{etale-cohomology-of-diamonds})
\[ \delta:\LSD_{K,\et}\to  \mathrm{SmRig}_{K,\et}.\]
We also need the analogous functor on the site of perfectoid spaces with the \'etale topology:
\[ \iota:\LSD_{K,\et}\to \Perf_{K,\et}.\]
Since for any adic space $Y$ over $K$ we have $Y_{\et}\cong Y^\diamondsuit_{\et}$, both $\delta_{\ast}$ and $\iota_{\ast}$ are exact functors.

For the definition of the diamantine Picard functor, we now consider the diamondification 
$\pi^{\diamondsuit}:X^{\diamondsuit}\to \mathrm{Spd}(K,K^+)$. Pullback along this map induces a natural morphism of sites
\[ \pi_{\et}^{\diamondsuit}:\LSD_{X,\et}\to \Perf_{K,\et}.\]

There is a second, much finer topology on perfectoid spaces and diamonds, namely the $v$-topology introduced by Scholze in \cite[\S8,\S14]{etale-cohomology-of-diamonds}. For this we get a morphism of sites
\[\pi_{v}^{\diamondsuit}:\LSD_{X,v}\to \Perf_{K,v}\]
 with the same underlying functor as $\pi_{\et}^{\diamondsuit}$, but finer topologies on either side.
\begin{Definition}
	The \'etale diamantine Picard functor of $X$ is defined to be the sheaf
	\[\uPic^{\diamondsuit}_{X,\et}:=R^1\pi^\diamond_{\et\ast}\G_m:\Perf_{K,\et}\to \mathrm{Ab}.\]
	In the diamantine setting, there is a second Picard functor defined using the finer v-topology:
	\[\uPic^{\diamondsuit}_{X,v}:=R^1\pi^\diamond_{v\ast}\G_m:\Perf_{K,v}\to \mathrm{Ab}.\]
\end{Definition}
\begin{Remark}
	The functor $\uPic^{\diamondsuit}_{X,\et}$ naturally extends to $\LSD_{X,\et}$ if we instead define $\pi^{\diamondsuit}_{\et}$ to be the morphism $\LSD_{X,\et}\to \LSD_{K,\et}$. We can then equivalently define $\uPic^{\diamondsuit}_{X,\et}:=\iota_{\ast}R^1\pi^{\diamondsuit}_{\et\ast}\G_m$. 
	The same works for $\uPic^{\diamondsuit}_{X,v}$, here the difference is less relevant as $\Perf_{K,v}$ is a basis of $\LSD_{K,v}$.
	
	That said, for our purposes it is important to restrict to perfectoid test objects: One reason is that $\G_m$ is easier to describe on adic spaces than on diamonds, another that relative $p$-adic Hodge theory is much simpler for morphisms $X\times Y\to Y$ over a perfectoid base $Y$.
\end{Remark}

The two functors $\uPic^{\diamondsuit}_{X,\et}$ and $\uPic^{\diamondsuit}_{X,v}$ are related via a natural morphism: One way to construct this is via the Leray sequences for the compositions in the commutative diagram
\[
\begin{tikzcd}
	{\text{LSD}_{X^\diamond,v}} \arrow[d] \arrow[r,"\pi^{\diamondsuit}_{v}"] \arrow[dotted,rd]& {\Perf_{K,v}} \arrow[d,"\nu"] \\
	{\text{LSD}_{X^\diamond,\et}} \arrow[r,"\pi^{\diamondsuit}_{\et}"]           & {\Perf_{K,\et}}.          
\end{tikzcd}\]
These induce a natural map
\[ \uP^{\diamondsuit}_{X,\et}\to \nu_{\ast}\uP^{\diamondsuit}_{X,v}.\]
Since both sides have the same underlying category $\Perf_K$, we shall in the following drop $\nu_{\ast}$ from notation, which amounts to forgetting that $\uP^{\diamondsuit}_{X,v}$ is already a sheaf for the $v$-topology.

\medskip

Second, the two functors $\uPic^{\diamondsuit}_{X,\et}$ and $\uPic^{\diamondsuit}_{X,v}$ are related to the rigid analytic Picard functor via a natural base-change map:
Consider the commutative diagram of big \'etale sites
\[
\begin{tikzcd}
	{\text{LSD}_{X^\diamond,\et}} \arrow[d] \arrow[r,"\delta"] & {\text{SmRig}_{X,\et}} \arrow[d,"\pi_{\et}"] \\
	{\text{LSD}_{K,\et}} \arrow[r,"\delta"]           & {\text{SmRig}_{K,\et}}.       
\end{tikzcd}\]

For the comparison, we wish to extend the diamondification functor from smooth rigid spaces to sheaves on $\text{SmRig}_{K,\et}$.
For any sheaf $\mathcal F$ on $\text{SmRig}_{K,\et}$, we therefore define
\[ \mathcal F^\diamondsuit:=\iota_{\ast}\delta^{-1}\mathcal F.\]
This notation makes sense since $\mathcal F^\diamondsuit$ agrees with $X^\diamond$ if $\mathcal F$ is represented by a smooth rigid space $X$. We note that   $-^\diamondsuit$ is exact because $\iota_{\ast}$ is. Consequently, the base change map for the above diagram induces 
for any $n\geq 0$ a natural morphism of sheaves on $\Perf_{K,\et}$
\begin{equation}\label{eq:base-change-from-rigid-to-diamond}
 (R^n\pi_{\et\ast}\mathcal F)^\diamondsuit\to R^n\pi^{\diamondsuit}_{\et\ast}\mathcal F^{\diamondsuit}.
\end{equation}
Applying this to $\mathcal F=\G_m$ and $n=1$, we find that there is a natural morphism 
\[ (\uP_{X,\et})^\diamondsuit\to \uP^{\diamondsuit}_{X,\et}\]
of sheaves on $\Perf_{K,\et}$.
It is clear from the construction that this is functorial in $X$.

\subsection{The Diamantine Picard Comparison Theorem}\label{s:diamantine-comparison-explained}
With the technical preparations of the last section, we can now formulate the most general version of our main result. As this can be interpreted as being ``relative $p$-adic Hodge theory for $\G_m$'', we first explain a related but simpler result, namely the case of $\G_a$.

\medskip

To simplify notation in the following without making additional choices, we begin by introducing notation for the usual Tate twist in $p$-adic Hodge theory:
\begin{Definition}
	For any smooth rigid space $X$ over $K$, we denote by $\wt \Omega^1_X:=\Omega^1_{X|K}\{-1\}$ the sheaf on $X_{\et}$ given by tensoring over $K$ with the Breuil--Kisin--Fargues twist $K\{-1\}$.  If $K$ contains all $p$-power unit roots, this twist is equivalent to the usual Tate twist $K(-1)$.
\end{Definition}
Assume from now on that $X$ is proper. As we will discuss,
following Scholze's approach to the Hodge--Tate sequence of $p$-adic Hodge theory \cite{Scholze_p-adicHodgeForRigid}\cite[\S3]{Scholze2012Survey}, one has in general:
\begin{Proposition}\label{p:HT-ses-vs}
There is a natural short exact sequence
\[
0\to H^1_{\an}(X,\O)\to H^1_{v}(X,\O)\to H^0(X,\wtOm_X^1)\to 0.\]
\end{Proposition}
This is canonically isomorphic to the usual Hodge--Tate short exact sequence if $K$ is algebraically closed, via the Primitive Comparison Theorem:  $ H^1_{v}(X,\O)=H^1_{\et}(X,\Q_p)\otimes_{\Q_p}{K}$.

As we will show, due to various base-change results for coherent cohomology, we have a relative ``diamantine'' version of this short exact sequence. For the formulation, we need:
\begin{Definition}
	For any abelian sheaf $F$ on $\Perf_{K,\et}$, we define its tangent space as \[T_0F:=\Hom(\G_a^+,F)\otimes_{\Z_p}\Q_p\] where $\G_a^+$ is the closed unit  with its additive structure. Since $\Hom(\G_a^+,\G_a^+)=K^+$, this is always a $K$-vector space. When $F$ is represented by a rigid group variety $G$, this recovers the usual notion of tangent spaces
	by \cite[Theorem~3.4]{heuer-G-torsors-perfectoid-spaces}.
\end{Definition}
\begin{Proposition}\label{p:HTses-of-R^1pi-for-O}
		Let $(K,K^+)$ be a perfectoid field extension of $\Q_p$. Let $\pi:X\to \Spa(K,K^+)$ be any proper smooth rigid space. Then:
		\begin{enumerate}
			\item The natural map $(R^1\pi_{\et\ast}\O)^\diamondsuit\to R^1\pi^\diamondsuit_{\et\ast}\O$ from \cref{eq:base-change-from-rigid-to-diamond} is an isomorphism.
			\item There is a short exact sequence of sheaves on $\Perf_{K,\et}$, functorial in $X$,
			\begin{equation}\label{seq:HT-ses}
				0\to R^1\pi^\diamondsuit_{\et\ast}\O\to R^1\pi^\diamondsuit_{v\ast}\O\xrightarrow{\HT} H^0(X,\wt\Omega^1_X)\otimes_K\G_a\to 0
			\end{equation}
		which is canonically isomorphic to $-\otimes_K\G_a$ applied to the sequence in \cref{p:HT-ses-vs}. In particular, we recover this sequence by passing to tangent spaces.
		\end{enumerate}
\end{Proposition}
We note that if $K$ is algebraically closed, this shows that $R^1\pi^\diamondsuit_{v\ast}\O=\Hom(\pi_1(X),K)\otimes \G_a$.

We can now formulate a precise version of the main result of this article, which could be described as a version of \cref{p:HTses-of-R^1pi-for-O} for the multiplicative group $\G_m$. For the statement, let us fix $\alpha\in \Q$ with $\alpha>\frac{1}{p-1}$ and $|p|^\alpha\in |K|$.  We recall that the $p$-adic exponential converges on the closed rigid analytic disc over $K$ of radius $|p|^{\alpha}$, which we shall simply denote by $p^{\alpha}\G_a^+$.
\begin{Theorem}[Diamantine Picard Comparison Theorem]\label{t:representability-of-v-Picard-functor}	Let $(K,K^+)$ be a perfectoid field extension of $\Q_p$ and let $X\to \Spa(K,K^+)$ be any proper smooth rigid space. Then:
	\begin{enumerate}
		\item The natural map $(\uP_{X,\et})^\diamondsuit\to \uP^{\diamondsuit}_{X,\et}$ is an isomorphism.
	\item There is a short exact sequence of abelian sheaves on $\Perf_{K,\et}$, functorial in $X$,
	\begin{equation}\label{seq:Picv-ses}
	0\to \uP_{X,\et}^\diamondsuit\to \uP_{X,v}^\diamondsuit\xrightarrow{\HT\log} H^0(X,\wt\Omega^1_X)\otimes_K\G_a\to 0.
	\end{equation}
	On tangent spaces at the identity, this recovers the sequence  \cref{seq:HT-ses}.
	\item The sequence becomes split over the bounded open subgroup of $H^0(X,\wt\Omega^1_X)\otimes_K\G_a$ defined as the image of $H^1_v(X,\O^+)\otimes p^\alpha\G_a^+$ under the Hodge--Tate map $\HT$ from \eqref{seq:HT-ses}.
	\end{enumerate}
 \end{Theorem}
\begin{Remark}
	We will see in \cite{heuer-geometric-Simpson-Pic} that the sequence \eqref{seq:Picv-ses} is never split  globally over all of $H^0(X,\wt\Omega^1_X)\otimes_K\G_a$ except in the trivial case of $H^0(X,\Omega_X^1)=0$. In fact, it is better to think about the morphism $\HT\log$ as a non-trivial $\uP^{\diamondsuit}_{X,\et}$-torsor for the \'etale topology. As we explain in detail in \cite{heuer-geometric-Simpson-Pic}, this perspective makes $\HT\log$ into an analogue of the Hitchin fibration.
\end{Remark}

Part 1 of the Theorem makes precise the idea  that in order to study \'etale line bundles on $X\times Y$ where $Y$ is a perfectoid space, it suffices to understand the situation for rigid $Y$, and vice versa. Part 2 is a geometric upgrade of \cite[Theorem~1.3.2]{heuer-v_lb_rigid} in the proper case and could be described as a statement about ``relative $p$-adic Hodge theory with $\G_m$-coefficients''.

\medskip

We already mention some consequences of \cref{t:representability-of-v-Picard-functor} that we will deduce in the end:
\begin{Corollary}\label{l:Pic_et-is-v-sheaf}
	\begin{enumerate}
\item $(\uP_{X,\et})^\diamondsuit$ is a v-sheaf on $\Perf_K$.
\item If we regard $\uP_{X,\et}^\diamondsuit$ as a functor on all of $\LSD_{K,\et}$,  then it also satisfies the sheaf property for $v$-covers $Y'\to Y$ where $Y$ is a smooth rigid space and $Y'$ is perfectoid.
\item If the rigid Picard functor $\uP_{X,\et}$ is represented by a rigid group $G$, then its diamondification $G^\diamondsuit$ represents $\uPic_{X,\et}^\diamondsuit$.
\item  Conversely, if there is a rigid group $G$ with $G^\diamond\cong \uPic_{X,\et}^\diamondsuit$, then  $G$ represents $\uPic_{X,\et}$.
\item $\uP^\diamondsuit_{X,\et}$ is represented by a rigid group if and only if $\uP^\diamondsuit_{X,v}$ is represented by a rigid group.
\end{enumerate} 
\end{Corollary}
\begin{Remark} 
	The first part shows in particular that the functor sending $Y\in\Perf_{K,v}$ to the groupoid of analytic line bundles on $X\times Y$ is a $v$-stack in the sense of \cite[\S9]{etale-cohomology-of-diamonds}. This is similar in spirit to the statement that vector bundles on the Fargues--Fontaine curve satisfy $v$-descent in the perfectoid variable (\cite[Proposition~II.2.1]{FS-Geometrization} \cite[Proposition~19.5.3]{ScholzeBerkeleyLectureNotes}).
\end{Remark}
\begin{Remark}
	Parts 2--5 might open up new strategies to prove that the rigid Picard functor is always representable by a rigid group whose identity component is semi-abeloid.
\end{Remark}

In fact, our proof for the \'etale comparison, i.e.\ part 1 of Theorem~\ref{t:representability-of-v-Picard-functor}, works also in higher cohomological degree. More precisely, our proof will give the following stronger statement:
\begin{Theorem}\label{p:higher-cohomology-diamond-base-change}
	Let $F$ be one of  $\O$, $\O^\times$, $\Z/N\Z$, $N\in\Z$. Then for $n\geq 0$,
	\[(R^n\pi_{\et\ast}F)^\diamond\isomarrow R^n\pi^{\diamond}_{\et\ast}F\]
	is an isomorphism, and both of these sheaves on $\Perf_{K,\et}$ are already $v$-sheaves.
\end{Theorem}
\begin{Remark}
	In \cite{heuer-sheafified-paCS}, we show that there is also a higher degree analogue of part 2 of Theorem~\ref{t:representability-of-v-Picard-functor}: an extension of \cref{seq:Picv-ses} to a spectral sequence in the category of abelian sheaves on $\Perf_{K,\et}$ which is a multiplicative analogue of the Hodge--Tate spectral sequence of $X$.
\end{Remark}
\subsection{Outline of proof strategy}\label{s:outline}

We now give an outline of the proof of Theorem~\ref{t:representability-of-v-Picard-functor}, essentially the same line of argument will show \cref{p:higher-cohomology-diamond-base-change}.
The basic strategy is to study step-by-step the following two cohomological diagrams of sheaves on $\Perf_{K,\et}$:
Following the notation in \cite{heuer-v_lb_rigid}, let us write $\O^\times$ for the sheaf of units on $\mathrm{SmRig}_{K,\et}$ represented by $\G_m$. Let \[\Oone:=1+\m\O^+\subseteq \O^\times,\]
where we recall that $\m$ denotes the ideal of topologically nilpotent elements of $K^+$. If $K^+=\O_K$, then $\Oone$ is represented by the open disc of radius $1$ around the origin, but in general it is the smaller open subgroup of $\G_m$ given by the union of closed discs of radius $<1$. The role of $\O^\times_1$ is that it is the domain of convergence of the $p$-adic logarithm $\log:\O^\times_1\to \G_a$.

 Let
\[\bOx:=\O^\times/\Oone\]
denote the quotient sheaf. As before, we shall identify these sheaves with their diamondifications, so that we obtain a short exact sequence on $\mathrm{SmRig}_{K,\et}$ as well as on $\Perf_{K,\et}$
\[ 0\to \Oone\to \O^\times\to \bOx\to 0.\]
Exactly as in \cite[Lemma~2.17]{heuer-v_lb_rigid}, we will see that $\bOx$ is in fact already a $v$-sheaf on $\Perf_K$.

 We now apply to the above sequence the two natural transformations
\[ (R^n\pi_{\et\ast}-)^\diamondsuit\to R^n\pi^\diamondsuit_{\et\ast}(-^\diamondsuit) \to R^n\pi^\diamondsuit_{v\ast}(-^\diamondsuit).\]
This results in a large commutative diagram of sheaves on $\Perf_{K,\et}$
\[
	\begin{tikzcd}[column sep = 0.25cm]
	{\pi^\diamondsuit_{v\ast}\bOx} \arrow[r] &{R^1\pi^\diamondsuit_{v\ast}\Oone}\arrow[r] &
		{R^1\pi^\diamondsuit_{v\ast}\O^\times} \arrow[r] & {R^1\pi^\diamondsuit_{v\ast}\bOx} \arrow[r]& {R^2\pi^\diamondsuit_{v\ast}\Oone} \\
	{\pi^\diamondsuit_{\et\ast}\bOx} \arrow[r]\arrow[u]  &{R^1\pi^\diamondsuit_{\et\ast}\Oone}\arrow[r]\arrow[u]  &
	{R^1\pi^\diamondsuit_{\et\ast}\O^\times} \arrow[r]\arrow[u]  & {R^1\pi^\diamondsuit_{\et\ast}\bOx} \arrow[r]\arrow[u] & {R^2\pi^\diamondsuit_{\et\ast}\Oone}\arrow[u]  \\
	{(\pi_{\et\ast}\bOx)^\diamond}\arrow[u]\arrow[r]&
	{(R^1\pi_{\et\ast}\Oone)^\diamond}\arrow[u] \arrow[r] &
	{(R^1\pi_{\et\ast}\O^\times)^\diamondsuit} \arrow[u] \arrow[r] &
	{(R^1\pi_{\et\ast}\bOx)^\diamondsuit} \arrow[u] \arrow[r] & {(R^2\pi_{\et\ast}\Oone)^\diamondsuit} \arrow[u]
\end{tikzcd}
\]
in which the bottom two rows are exact with respect to the \'etale topology and the top row is exact with respect to the $v$-topology (i.e.\ we have tacitly applied $\nu_{\ast}$ to the top row).

We can without loss of generality assume that $X$ is connected. The first step of the proof can then be summarised by saying that we will prove:
\begin{Lemma}\label{l:describing-master-diagram-1}
	The following hold in the above commutative diagram:
	\begin{enumerate}[label=(\Alph*)]
		\item The leftmost horizontal transition maps are $0$.
		\item In the fourth column, both maps are isomorphisms.
		\item In the fifth column, the composition of the vertical maps is injective.
	\end{enumerate}
\end{Lemma}
Once this is achieved, it follows formally that the top row is already exact for the \'etale topology: Indeed, exactness at the second and third term follows from (A) using that $\nu_\ast$ is left-exact, and exactness at the fourth term follows from (B) and (C) by a diagram chase.

 At this point, the 5-Lemma (applied once to the bottom maps and once to the compositions) reduces us to proving a variant of the Theorem for $\Oone$ instead of $\G_m$:
\begin{Proposition}\label{p:relative-HT-for-1+m}\leavevmode
	\begin{enumerate}
		\item The map $(R^1\pi_{\et\ast}\Oone)^\diamondsuit\to R^1\pi^\diamondsuit_{\et\ast}\Oone$ is an isomorphism.
		\item There is a short exact sequence of abelian sheaves on $\Perf_{K,\et}$
	\[ 0\to R^1\pi^\diamondsuit_{\et\ast}\Oone\to R^1\pi^\diamondsuit_{v\ast}\Oone\xrightarrow{\HT\log} H^0(X,\wt\Omega^1_X)\otimes_K\G_a\to 0.\]
\end{enumerate}
\end{Proposition}
In order to prove this, we apply the same strategy as above to the logarithm sequence 
\[ 0\to \mu_{p^\infty}\to \Oone\xrightarrow{\log} \O\to 0,\]
which results in a commutative diagram of sheaves on  $\Perf_{K,\et}$
\begin{equation}\label{eq:master-dg-2}
\begin{tikzcd}[column sep = 0.25cm]
{\pi^\diamondsuit_{v\ast}\O} \arrow[r] &{R^1\pi^\diamondsuit_{v\ast}\mu_{p^\infty}} \arrow[r]           & {R^1\pi^\diamondsuit_{v\ast}\Oone} \arrow[r]           & {R^1\pi^\diamondsuit_{v\ast}\O} \arrow[r]           & {R^2\pi^\diamondsuit_{v\ast}\mu_{p^\infty}} \\
{\pi^\diamondsuit_{\et\ast}\O} \arrow[r] \arrow[u]  &{R^1\pi^\diamondsuit_{\et\ast}\mu_{p^\infty}} \arrow[r] \arrow[u]            & {R^1\pi^\diamondsuit_{\et\ast}\Oone} \arrow[r] \arrow[u]            & {R^1\pi^\diamondsuit_{\et\ast}\O} \arrow[r] \arrow[u]           & {R^2\pi^\diamondsuit_{\et\ast}\mu_{p^\infty}}  \arrow[u] \\
{(\pi_{\et\ast}\O)^\diamondsuit} \arrow[u] \arrow[r] &{(R^1\pi_{\et\ast}\mu_{p^\infty})^\diamondsuit} \arrow[u] \arrow[r] & {(R^1\pi_{\et\ast}\Oone)^\diamondsuit} \arrow[u] \arrow[r] & {(R^1\pi_{\et\ast}\O)^\diamondsuit} \arrow[u] \arrow[r] & {(R^2\pi_{\et\ast}\mu_{p^\infty})^\diamondsuit} \arrow[u]
\end{tikzcd}
\end{equation}
in which again the bottom two rows are exact with respect to the \'etale topology and the top row is exact with respect to the $v$-topology. The fourth row of this diagram is described by \cref{p:HTses-of-R^1pi-for-O}. We use this description to prove:

\begin{Lemma}\label{l:describing-master-diagram-2}
	The following hold in the above commutative diagram:
\begin{enumerate}[label=(\Alph*)]
	\setcounter{enumi}{3}
	\item The leftmost horizontal transition maps are $0$.
	\item In the second and fifth column, all maps are isomorphisms.
	\item Further towards the right, the map $(R^2\pi_{\et\ast}\O)^\diamondsuit\to (R^2\pi_{v\ast}\O^\diamondsuit)$ is injective.
\end{enumerate}
\end{Lemma}
Using parts (D), (E), and a direct argument describing the vertical cokernels in the middle of the top two rows, we see that all sheaves in the above diagram are in fact \mbox{$v$-sheaves}, so that may regard \eqref{eq:master-dg-2} as a commutative diagram of $v$-sheaves with exact rows.
Part (F) then implies part (C) above. Parts (D) and (E) will be enough to prove part 1 of Proposition~\ref{p:relative-HT-for-1+m}.1, and left-exactness in Proposition~\ref{p:relative-HT-for-1+m}.2. Finally, we will use a relative version of the argument in \cite[\S3.5]{heuer-v_lb_rigid} to prove the right-exactness, using the diamantine universal cover.  This will complete the proof of the Theorem.

\medskip

From all of the above steps, the key step is arguably the proof of (B): This is where we need to study the transition from functors on smooth rigid spaces to functors on perfectoid spaces in great detail. We do this by proving a very general rigid approximation lemma. This will also be handy to complete some of the other steps, although only (B) uses it in its full force. Proving the rigid approximation lemma is the goal of the next section.

\section{A rigid approximation lemma}
We now begin the proof of \cref{t:representability-of-v-Picard-functor} with a rigid approximation lemma for the sheaf $\bOx=\O^\times/\Oone$ introduced in \cref{s:outline}. This will be required for step (B). 
For the statement we use tilde-limits \cite[(2.4.1)]{huber2013etale}\cite[\S2.4]{ScholzeWeinstein} as well as a slight strengthening for affinoids:
\begin{Definition}
	For a cofiltered inverse system of adic spaces $(X_i)_{i\in I}$ with qcqs transition maps, and an adic space $X_\infty$ with compatible maps $X_\infty\to X_i$ for all $i\in I$, we write
	\[ X_\infty\sim \varprojlim_{i\in I} X_i\]
	if on the underlying topological spaces, the maps induce a homeomorphism $|X_\infty|=\varprojlim |X_i|$, and if there is a cover of $X_\infty$ by affinoid opens $U_\infty$ for which the map $\varinjlim_{U}\O(U)\to \O(U_\infty)$
	has dense image, where $U\subseteq X_i$ runs through all affinoid opens through which $U_\infty\to X_i$ factors, and all $i$. If moreover all $X_i$ and $X_\infty$ are affinoid, we write
	\[ X_\infty\approx \varprojlim_{i\in I} X_i\]
	if already the global sections $\varinjlim \O(X_i)\to \O(X_\infty)$ have dense image.
\end{Definition}

\begin{Proposition}\label{p:rigid-approxi-for-bOx}
	Let $Y$ be an affinoid perfectoid space over $K$ and let $(Y_i)_{i\in I}$ be a cofiltered inverse system of affinoid smooth rigid spaces such that $Y\approx \varprojlim Y_i$. Let $X$ be a qcqs adic space over $K$ that is either smooth or perfectoid. Let $U_i\to X\times Y_i$ be a qcqs \'etale map and set $U_j:=U_i\times_{Y_i}Y_j$ and $U:=U_i\times_{Y_i}Y$. Then for all $n\geq 0$, the natural map 
	\[\varinjlim_{j\geq i}H^n_{\et}(U_{j},F)\to H^n_{\et}(U,F)\]
	is an isomorphism
	for $F=\bOx$ or $F=\O^+/\varpi$  for any $0\neq \varpi\in \m$.
\end{Proposition}

\begin{Remark}\label{rm:XxY-rep}
	\begin{itemize}
		\item A priori, the fibre product $X\times Y_i$ is in the category of diamonds over $\Spd(K,K^+)$. But since $Y_i$ is smooth, this is represented by a sousperfectoid adic space in the sense of Hansen--Kedlaya \cite[\S6.3]{ScholzeBerkeleyLectureNotes}: the fibre product of $X$ and $Y_i$ over $\Spa(K)$ in the category of uniform adic spaces.
	\item The analogue of the proposition for $F= \O^+$ and $\O^\times$ fails already for $n=0$.
	\item With some more work, the assumption of the Proposition can be weakened, e.g.\ it also holds in characteristic $p$:	 The proof for perfectoid $X$ works without changes. For rigid $X$, one can then use local sections of Frobenius to descend from the perfection. 
	\end{itemize}
\end{Remark}
\begin{proof}	
	The proof will be completed by a series of lemmas. We start with an easy observation:
\begin{Lemma}
	In the situation of Proposition~\ref{p:rigid-approxi-for-bOx}, we have:
	\begin{enumerate}
	\item $U=\varprojlim_{j\geq i} U_j$ as diamonds.
	\item If $U\to X\times Y$ is an \'etale cover, then so is $U_j\to X\times Y_j$ for $j\gg i$.
	\end{enumerate}
\end{Lemma}
	\begin{proof}
		We have $Y^\diamond=\varprojlim_{i\in I} Y_i^\diamond$ by \cite[Proposition~2.4.5]{ScholzeWeinstein}. Part (i) follows since limits commute with fibre product.
	
	Part 2 follows from 1 due to the qcqs assumption: Namely, since $U_j\to X\times Y_j$ is \'etale, it is open \cite[Proposition 1.7.8]{huber2013etale}, so we can without loss of generality replace $U_j$ by its quasi-compact open image. The statement now follows from the following Lemma.
	\end{proof}

	\begin{Lemma}[{\cite[Lemma~6.13.(iv)]{perfectoid-spaces}}]\label{l:covers-come-are-covers-for-J>>I}
		Let $T=\varprojlim T_i$ be a cofiltered inverse limit of spectral spaces with spectral transition maps. Let $U\subseteq T_i$ be a quasi-compact open such that $T\to T_i$ factors through $U$. Then already some $T_j\to T_i$ factors through $U$.
	\end{Lemma}
\begin{proof}
	For any $q:T_j\to T_i$ in the inverse system, set $Z_j:=T_j\backslash q^{-1}(U)$. Then the assumptions imply $\varprojlim Z_j=\emptyset$. The desired result now follows from  \cite[0A2W]{StacksProject}.
\end{proof}
	
	In order to continue, we need in the following several subcategories of the \'etale site: 
	\begin{Definition}
		\begin{enumerate}
			\item Let $Z$ be a locally spatial diamond over $K$. Let $Z_{\etqcqs}\subseteq Z_{\et}$ be the full subcategory of quasi-compact quasi-separated \'etale morphisms $U\to Z$. For any adic space $Z$, this also defines $Z_{\etqcqs}$  via the identification $Z_\et=Z_\et^\diamond$.
			\item 	
			If $Z$ is an affinoid adic space over $K$, let $Z_{\stdet}\subseteq Z_{\et}$ be the full subcategory of objects $Z'\to Z$ which are successive compositions of rational open immersions and finite \'etale maps. We call such maps standard-\'etale. By \cite[Lemmas~11.31 and 15.6]{etale-cohomology-of-diamonds}, these form a basis of $Z_{\et}$. Note that we have $Z_{\stdet} \subseteq Z_{\etqcqs}\subseteq Z_{\et}$.
		\end{enumerate}
	
	\end{Definition}
	
	Next, we explain that in order to prove Proposition~\ref{p:rigid-approxi-for-bOx} for all $n\geq 0$, we can reduce to the case of $n=0$. We first note that  for $n=0$, the statement is the following:
	\begin{claim}\label{cl:approx-O^+/p-for-U-sheaf-version}
		In the situation of \Cref{p:rigid-approxi-for-bOx}, for any $U_i\in (X\times Y_i)_{\etqcqs}$ with pullbacks $U_j\to X\times Y_j$ and $U\to X\times Y$, we have
		\[\bOx(U)= \varinjlim_{j\geq i} \bOx(U_j),\quad\quad\O^+/\varpi(U)= \varinjlim_{j\geq i} \O^+/\varpi(U_j) .\]
	\end{claim}
	Suppose that Claim~\ref{cl:approx-O^+/p-for-U-sheaf-version} holds true. Then the case of $n\geq 0$ follows from very general results on cohomology in inverse limit topoi:
	\begin{Lemma}\label{l:cohom-of-limit-topos}
		Let $Z=\varprojlim_{i\in I} Z_i$ be a cofiltered inverse limit of spatial diamonds over $\Spa(K,K^+)$. Let $F$ be an abelian sheaf on $\LSD_{K,\et}$. Assume that for all $i\in I$ and $U_i\in Z_{i,\etqcqs}$ with pullbacks $U_j=U_i\times_{Z_i}Z_j$ and $U=U_i\times_{Z_i}Z_\infty$ we have 
		\[ F(U)=\varinjlim_{j\geq i} F(U_j).\]
		Then for all $n\geq 0$,
		\[H^n(U,F)=\varinjlim_{j\geq i} H^n(U_j,F).\]
	\end{Lemma}
	\begin{proof}
	Since the $Z_i$ are spatial, we have by \cite[Proposition~11.23]{etale-cohomology-of-diamonds} an equivalence of sites $Z_{\etqcqs}=\twolim_iZ_{i,\etqcqs}$. Write $\mu_j:Z\to Z_j$ for the natural projection, then by \cite[09YN]{StacksProject} our assumptions imply that $F_{Z}=\varinjlim_{j\geq i} \mu_{j}^{-1}F_{Z_j}$, from which the statement follows formally by \cite[VI Th\'eor\`eme~8.7.3]{SGA4}  or \cite[09YP]{StacksProject}. In fact, we will later only need the case  $n=1$, in which case this is a simple \cH\ argument.
	\end{proof}
	We have thus reduced \cref{p:rigid-approxi-for-bOx} to Claim~\ref{cl:approx-O^+/p-for-U-sheaf-version}. We now first prove Claim~\ref{cl:approx-O^+/p-for-U-sheaf-version} for qcqs perfectoid $X$. In this case, we can further reduce the claim to the following statement:
	\begin{claim}\label{cl:approx-O^+/p-for-U}
		In the situation of Proposition~\ref{p:rigid-approxi-for-bOx},
		 assume further that $X$ is affinoid perfectoid. Then for any $U_i\in (X\times Y_i)_{\stdet}$ with pullbacks $U_j\in (X\times Y_j)_{\stdet}$ for $j\geq i$ and $U\in (X\times Y)_{\stdet}$, we have		\[\O^\times(U)/\Oone(U)= \varinjlim_{j\geq i} \O^\times(U_j)/\Oone(U_j),\quad \quad
		\O^+(U)/\varpi= \varinjlim_{j\geq i} \O^+(U_j)/\varpi.\]

	\end{claim}
	Indeed, suppose we know Claim~\ref{cl:approx-O^+/p-for-U}. Then using the equivalence of sites \[(X\times Y)_{\stdet}=\twolim_i(X\times Y_i)_{\stdet}\]
	and the fact that $(X\times Y)_{\stdet}$ is a basis for $(X\times Y)_{\et}$, it follows upon sheafification that
	\[\bOx(U)= \varinjlim_{j\geq i} \bOx(U_j),\quad \O^+/\varpi(U)= \varinjlim_{j\geq i} \O^+/\varpi(U_j).\]
	More generally, we now also obtain these last two equalities if $U_i\in (X\times Y_i)_{\etqcqs}$, because any such $U_i$ can be covered by finitely many objects of  $(X\times Y_i)_{\stdet}$ and their intersections are again covered by finitely many objects of  $(X\times Y_i)_{\stdet}$ due the qcqs assumption.
	
	In a second step, this now implies Claim~\ref{cl:approx-O^+/p-for-U-sheaf-version} for qcqs perfectoid $X$, by applying the same covering argument to a finite cover of $X$ by affinoid perfectoid subspaces.
	
	\medskip
	
	We now prove Claim~\ref{cl:approx-O^+/p-for-U} step by step. We first treat the case that $X=\Spa(K,K^+)$ where $(K,K^+)$ is a perfectoid field. In fact, for the following discussion until \Cref{l:part-2-for-fet-on-Y} inclusively, we can allow the greater generality that  $(K,K^+)$ is any non-archimedean field of residue characteristic $p$. We fix a pseudo-uniformiser $0\neq \varpi\in \m$.

\begin{Lemma}\label{l:dense-image-implies-isomorphism-on-O^+/varpi}
	Let $(Y_i)_{i\in I}$ be a cofiltered inverse system of affinoid adic spaces over $K$ with an affinoid tilde-limit $Y\approx \varprojlim Y_i$. Then the following maps are isomorphisms:
	\[ \varinjlim_i \O^\times(Y_i)/\Oone(Y_i)\to \O^\times(Y)/\Oone(Y),\quad\quad\varinjlim_i \O^+(Y_i)/\varpi\to \O^+(Y)/\varpi.\]
\end{Lemma}
\begin{proof}
	Let $f\in \O^+(Y)$. Then we approximate $f$ by some $f_i\in \O(Y_i)$ whose image in $\O(Y)$ satisfies $|f_i-f|\leq |\varpi|$. In particular, we then have $f_i\in \O^+(Y)$. The condition $|f_i|\leq 1$ defines a quasi-compact open subspace $U$ of $Y_i$ through which $Y\to Y_i$ factors. We may apply  Lemma~\ref{l:covers-come-are-covers-for-J>>I} to this situation since any morphism between affinoid analytic adic spaces is spectral. Consequently, there is $j\geq i$ such that $Y_j\to Y_i$ factors through $U$, which means that the image $f_j$ of $f_i$ in $\O(Y_j)$ is already in  $\O^+(Y_j)$. This shows surjectivity.
	
	Injectivity follows by a similar argument: If $f_i\in \O^+(Y_i)$ is such that its image $f\in \O^+(Y)$ is already in $\varpi\O^+(Y)$, then some $Y_j\to Y_i$ factors through the quasi-compact open defined by $|f_i|\leq |\varpi|$ because $Y\to Y_i$ does. Thus $f_i$ goes to $0\in \O^+(Y_j)/\varpi$.
	
	The proof for $\O^\times$ is similar as the proof of \cite[Lemma~2.17]{heuer-v_lb_rigid} goes through verbatim: To see injectivity, let $g_i\in \Oone(Y_i)$ be in the kernel. Then since $Y$ is quasi-compact, there is $\epsilon>0$ such that $|g_i-1|\leq |\varpi|^\epsilon$ on $Y$. By Lemma~\ref{l:covers-come-are-covers-for-J>>I} we have $g_i\in \Oone(Y_j)$ for some $j\gg i$.
	
	To see that the map is surjective, let $f\in \O^\times(Y)$. By assumption, we can find approximating sequences $f_i\to f$ and $f'_i\to f^{-1}$ with $f_i,f_i'\in \O(Y_i)$. Then $f_if'_i\to 1$ and thus $f_if'_i\in \Oone(Y)$ for $i\gg 0$. By the above argument, it follows that $f_if'_i\in \Oone(Y_j)$ for some $j\geq i$. But then $f_i\in \O^\times(Y_j)$, which implies that $f_i$ is in the image of the map.
\end{proof}
	According to \cref{l:dense-image-implies-isomorphism-on-O^+/varpi},
	in order to prove \cref{cl:approx-O^+/p-for-U}, it suffices to prove that $U\approx \varprojlim U_i$. We now first prove this when $U_i=X\times Y_i$. Recall that $X$ is affinoid perfectoid in \cref{cl:approx-O^+/p-for-U}.
	
	\begin{Lemma}\label{l:part-2-for-XxY}
		Let $X$ and $Y$ be affinoid perfectoid spaces and assume we have a tilde-limit $Y\approx \varprojlim Y_i$ for some smooth affinoid rigid spaces $Y_i$ over $K$. Then $X\times Y\approx \varprojlim X\times Y_i$.
	\end{Lemma}

	\begin{proof}
		Since $X$ and $Y$ are affinoid perfectoid, we have
		\[\O^+(X\times Y)/\varpi\aeq \O^+(X)\otimes_{K^+}\O^+(Y)/\varpi.\]
		The diamond $X\times Y_i$ is represented by an affinoid adic space (cf \cref{rm:XxY-rep}): This is defined by the Huber pair $(B_i[1/\varpi],B_i)$ given by setting
		$A_i:=\O^+(X)\hotimes\O^+(Y_i)$
		and defining $B_i$ to be the integral closure of the image of $A_i$ in $A_i[1/\varpi]$. 
		Consider now the composition
		\[ \varinjlim A_i\to \varinjlim B_i\to \O^+(X\times Y)\aeq \O^+(X)\hotimes_{K^+}\O^+(Y).\]
		Here the second map is $\O^+$ evaluated on $X\times Y\to X\times Y_i$. We wish to see that this second map is an almost isomorphism mod $\varpi$. This will imply that $X\times Y\approx \varprojlim_iX\times Y_i$.
		
		By Lemma~\ref{l:dense-image-implies-isomorphism-on-O^+/varpi}, the assumptions imply that $\O^+(Y)/\varpi=\varinjlim \O^+(Y_i)/\varpi$, hence the above composition is an almost isomorphism mod $\varpi$. The statement now follows formally: Let us axiomatise the argument for later reference.
	\begin{Lemma}\label{l:formal-argument-for-isomorphism}
		Let $A\xrightarrow{f} B \xrightarrow{g} C$ be morphisms of $\O_K^a$-modules such that
		\begin{itemize}
			\item $B$ is $\varpi$-torsionfree
			\item $f[1/\varpi]$ and $(g\circ f)/\varpi$ are both isomorphisms.
		\end{itemize}
		Then $g/\varpi$ is an isomorphism.
	\end{Lemma}
	\begin{proof}
		Let $T\subseteq A$ be the $\varpi$-power torsion submodule, then $T/\varpi\hookrightarrow A/\varpi$ is injective.
		The map $f/\varpi$ is almost injective since $(g\circ f)/\varpi$ is. On the other hand, $T\to B$ is trivial since $B$ is $\varpi$-torsionfree. This implies that $T/\varpi=0$. In particular, we have $(A/T)/\varpi=A/\varpi$. We also have  $(A/T)[1/\varpi]=A[1/\varpi]$. We may thus replace $A$ by $A/T$ and assume without loss of generality that $A$ is $\varpi$-torsionfree. In particular, $f$ is then injective since $f[1/\varpi]$ is.
		
		In this situation, the cokernel of $f$ is $\varpi$-power torsion (since $f[1/\varpi]$ is an isomorphism) but also $\varpi$-torsionfree (since $f/\varpi$ is injective). Consequently, $f$ is an isomorphism, thus $f/\varpi$ is an isomorphism, and hence so is $g/\varpi$ given that $(g\circ f)/\varpi$ is.
	\end{proof}
	Applying \cref{l:formal-argument-for-isomorphism} to the given sequence
	finishes the proof of Lemma~\ref{l:part-2-for-XxY}.
\end{proof}
	Finally, in order to prove \Cref{cl:approx-O^+/p-for-U}, it remains to add a standard-\'etale map on top of $X\times Y_i$. Setting $Z:=X\times Y$ and $Z_i:=X\times Y_i$ to simplify notation, we wish to see:
	\begin{Lemma}\label{l:part-2-for-fet-on-Y}
		Let $Z\approx \varprojlim Z_i$ be an affinoid perfectoid tilde-limit of affinoid adic spaces $Z_i$ over $K$. Let $U_i\to Z_i$ be an object of $Z_{i,\stdet}$. For $j\geq i$ write $U_j:=U_i\times_{Z_i}Z_j$ and   $U:=U_i\times_{Z_i}Z$. If all of these are adic spaces, then $U\approx \varprojlim_{j\geq i} U_j$.
	\end{Lemma}
	\begin{proof}
	We can prove this separately in the cases of finite \'etale maps and rational localisation.
		\paragraph{Case 1: $U_i\to Z_i$ finite \'etale.}
		Write $S_j=\O(Z_j)$ and $S:=\O(Z)$.  Similarly, for any $j\geq i$, let  $R_j:=\O(U_j)$ and $R:=\O(U)$.
		By \cite[Lemma 15.6]{etale-cohomology-of-diamonds} and \cite[Lemma~8.2.17.(i)]{KedlayaLiu-rel-p-p-adic-Hodge-I}, the map $S_i\to R_i$ is finite \'etale and we have
	$R_j=R_i\otimes_{S_i}S_j$ and $R=R_i\otimes_{S_i}S$. To see that 	$\varinjlim_{j\geq i}R_j\to R$ has dense image, it now suffices to see that we can approximate simple tensors $r\otimes s$ in $R_i\otimes_{S_i}S=R$. For this we use that $\varinjlim \O(S_j)\to \O(S)$ has dense image to find a sequence $s_j\in S_j$ with $s_j\to s$. This implies $r\otimes s_j\to  r\otimes s$, showing the desired dense image property.
	By \cite[Remark 2.4.3.(ii)]{huber2013etale}, we also have $|U|=\varprojlim|U_j|$. Thus $U\approx \varprojlim U_j$.

	\paragraph{Case 2: $U_i\to Z_i$ a rational open immersion.}
	 Since $Z$ is affinoid perfectoid, we have $S^+:=\O^+(U)\aeq \O^\circ(U)=S^\circ$. It therefore suffices to prove that for all $n\in \N$,
	 \[ \varinjlim \O^+(U_j)/\varpi^n\to \O^\circ(U)/\varpi^n\]
	 is an almost isomorphism. Second, by \cite[Lemma 6.4]{perfectoid-spaces}, the rational open $U\subseteq Z$ is of the form $Z(f_1,\dots,f_n/g)$ for some $f_1,\dots,f_n,g\in \sharp(\O^{\flat+}(Z))$ with $f_n=\varpi^N$ for some $N$, and can be written as
	\[ U=\Spa(R,R^+) \quad \text{where} \quad R^+\aeq S^\circ\langle (f_1/g)^{1/p^\infty},\dots,(f_n/g)^{1/p^\infty}\rangle.\]
	
	Set $\epsilon:=|\varpi^N|$. Then for every $l\in \N$, we can find $j_l$ large enough such that there are $f_{1,l},\dots,f_{n,l},g_l$ in $\O(Z_{j_l})$ such that on $Z$ we have for all $i=1,\dots,n$:
	\[|f_{i,l}-f_i^{1/p^l}|\leq \epsilon \quad \text{ and }\quad |g_{l}-g^{1/p^l}|\leq \epsilon.\]
	 Then on $Z$, the conditions
	$|f_{i,l}|\leq |g_l|$ and $|f_{i}^{1/p^l}|\leq |g^{1/p^l}|$
	are equivalent, and thus cut out the same rational open subspace of $Z$. For any $j\geq j_l$, let \[U_{j}:=Z_{j}(f_{1,l},\dots,f_{n,l}/g_l),\]
	then this means that $U=Z\times_{Z_j}U_{j}$. Moreover, we have a natural isomorphism \[S^\circ\langle (f_i/g)^{1/p^l}\rangle/\varpi=S^\circ\langle f_{i,l}/g_l\rangle/\varpi\] given explicitly by
	\[ S^\circ/\varpi[T_1\dots,T_n]/(T_ig^{1/p^l}-f_i^{1/p^l})\isomarrow S^\circ/\varpi[T_1,\dots,T_n]/(T_ig_l-f_{i,l}),\]
	\[ T_i\mapsto \Big(1+\frac{g^{1/p^l}-g_l}{g_l}\Big)^{-1}\Big(T_i-\frac{f_{i,l}-f_i^{1/p^l}}{g_l}\Big).\]
	Under these compatible identifications, in the limit over $l$, it makes sense to write
	\[R^{\circ}/\varpi\aeq S^\circ/\varpi\Big[\tfrac{f_{i,l}}{g_l}\big|i=1,\dots,n \text{ and }l\in \N\Big].\]

	For fixed $l$, Lemma~\ref{l:covers-come-are-covers-for-J>>I} implies that for $j\gg j_l$ we have $f_{i,l},g_l\in \O^+(Z_{j})$. Let \[A_{j,l}:=\O^+(Z_{j})\langle f_{i,l}/g_{l}|i=1,\dots,n\rangle,\quad B_{j}:=\O^+(U_j).\]
	Explicity, $B_{j}$ is the integral closure of the image of $A_{j,l}$ in $A_{j,l}[\tfrac{1}{\varpi}]$. In particular, we have 
	\[B_j[1/\varpi]=A_{j,l}[1/\varpi].\]
		
	We now observe that by construction, for any fixed $l$, the map
	\[\varinjlim_{j\geq j_l} A_{j,l}/\varpi= \varinjlim_{j\geq j_l}(\O^+(Z_{j})/\varpi)[f_{i,l}/g_{l}]\to (S^\circ/\varpi)[f_{i,l}/g_{l}]\aeq (S^\circ/\varpi)[(f_i/g)^{1/p^l}]\]
	is an almost isomorphism since this was true before tensoring with $\O^+(Z_{j_l})[f_{i,l}/g_l]$. Taking the colimit over $l$, this shows that also 
	 \[\varinjlim_{l\in \N}\varinjlim_{j\geq j_l} A_{j,l}/\varpi\aeq (S^\circ/\varpi)[(f_i/g)^{1/p^\infty}]\aeq  \O^+(U)/\varpi\]
	 is an almost isomorphism. The desired statement now follows from Lemma~\ref{l:formal-argument-for-isomorphism} applied to the sequence 
	\[ \varinjlim_l \varinjlim_{j\geq j_l} A_{i,l}\to \varinjlim_j B_j\to \O^+(U)/\varpi.\]
	This finishes the proof of Lemma~\ref{l:part-2-for-fet-on-Y}.
	\end{proof}

Combining \cref{l:part-2-for-fet-on-Y} with \cref{l:dense-image-implies-isomorphism-on-O^+/varpi}, we  have thus proved Claim~\ref{cl:approx-O^+/p-for-U}, which finishes the proof of Proposition~\ref{p:rigid-approxi-for-bOx} for perfectoid $X$.

In order to deduce the case of smooth rigid $X$, it again suffices to prove Claim~\ref{cl:approx-O^+/p-for-U-sheaf-version} in this case. To this end, we return to our original setup that $K$ is a perfectoid field over $\Q_p$. We then have the following consequence of the perfectoid case:
\begin{Lemma}\label{l:bOx-et-vs-v-on-XxY}
	Let $X$ be a smooth rigid space and let $Y$ be affinoid perfectoid over $K$. Then on $X\times Y$, we have
	\[ \O^+_{\et}/\varpi \aeq \nu_{\ast}(\O^+_{v}/\varpi).\]
	In particular, this holds on any smooth rigid space. Similarly, $\bOx_{\et} = \nu_{\ast}\bOx_v$.
\end{Lemma}
\begin{Remark} We show a  more general statement in \cite[Proposition~2.14]{heuer-G-torsors-perfectoid-spaces}, which also says that in fact, we can write $=$ instead of $\aeq$. But \cref{l:bOx-et-vs-v-on-XxY} is much easier to see:
\end{Remark}
\begin{proof}
	The statement is local on $X$. We may therefore assume that $X$ is affinoid and that we can find an affinoid perfectoid pro-finite-\'etale Galois cover $\wt X=\varprojlim X_i\to X$ with group $G=\varprojlim G_i$. Let $U\in (X\times Y)_{\stdet}$ be standard-\'etale, in particular affinoid, and let $\wt U\to \wt X\times Y$ be the pullback. We can without loss of generality assume that $\wt U=\varprojlim U_i\to U$ is affinoid perfectoid and pro-finite-\'etale Galois with group $G$. Then 
	we have $(\O^+_{\et}/\varpi)(\wt U)\aeq (\O^+_{v}/\varpi)(\wt U)$ since $\wt U$ is perfectoid and $\O^+$ is almost acyclic on affinoid perfectoids for both the \'etale and the $v$-topology by \cite[Proposition~7.13]{perfectoid-spaces} and \cite[Proposition~8.8]{etale-cohomology-of-diamonds}, respectively. Consequently,
	\[(\O^+_{v}/\varpi)(U)=(\O^+_{v}/\varpi)(\wt U)^G\aeq  (\O^+_{\et}/\varpi)(\wt U)^G=\varinjlim_i( \O^+_{\et}/\varpi)(U_i)^{G_i}=\O^+_{\et}/\varpi(U),\]
	where the third step follows from  \cref{l:part-2-for-fet-on-Y} and \cref{l:dense-image-implies-isomorphism-on-O^+/varpi} upon \'etale  sheafification.
	
	The case of $\bOx$ is analogous once we know that
	\[\bOx_{\et}(\wt U)= \bOx_{v}(\wt U).\]
	To see this, let $\tau$ be either the \'etale or the v-topology, then we have the exponential sequence from \cite[Lemma~2.18]{heuer-v_lb_rigid}:
	\[0\to \O_\tau \xrightarrow{\exp} \varinjlim_{x\mapsto x^p }\O^\times_\tau \to \bOx_{\tau}\to 1.\]
	This remains short exact after evaluating at $\wt U$ since $\O$ is acyclic on affinoid perfectoids in both topologies. The desired statement now follows from the fact that $\O_\et(\wt U)=\O_v(\wt U)$ and $\O^\times_\et(\wt U)=\O^\times_v(\wt U)$ since $\wt U$ is perfectoid.
\end{proof}
	
It follows that in order to prove Claim~\ref{cl:approx-O^+/p-for-U-sheaf-version} for rigid $X$, it suffices to prove the statement for $\O^+_{\et}/\varpi$ replaced by $\O^+_{v}/\varpi$, and $\bOx_{\et}$ replaced by $\bOx_{v}$. But since $X$ is a qcqs smooth rigid space, there is a $v$-cover of $X$ by a qcqs perfectoid space $\wt X$ such that $\wt X\times_X\wt X$ is qcqs perfectoid. Therefore the result now follows from the statement for perfectoid $X$.

 This finishes the proof of Proposition~\ref{p:rigid-approxi-for-bOx}.
\end{proof}

Our main application of  Proposition~\ref{p:rigid-approxi-for-bOx} is that it implies the first part of Lemma~\ref{l:describing-master-diagram-1}.(B): 
\begin{Corollary}\label{c:Lemma-B}
	Let $\pi:X\to \Spa(K)$ be a qcqs smooth rigid space. Then the morphism of sheaves on $\Perf_{K,\et}$
	\[(R^n\pi_{\et\ast}\bOx)^\diamond\isomarrow R^n\pi^{\diamond}_{\et\ast}\bOx\]
	from \eqref{eq:base-change-from-rigid-to-diamond}
	is an isomorphism for all $n\geq 0$. Similarly for $\O^+/\varpi$ for any $0\neq \varpi\in K^+$.
\end{Corollary}
\begin{proof}
	Unravelling the definition, we see that both sides are the \'etale sheafifications of the presheaves on $\Perf_K$ defined as follows: The left hand side is
	\[ Y\mapsto \varinjlim_{Y\to Z}H^n_{\et}(X\times Z,\bOx)\]
	where $Y\to Z$ ranges through all morphisms to affinoid smooth rigid spaces  $Z$ over $K$, and where we use $Y_{\etqcqs}=\twolim Z_{\etqcqs}$ to see that it suffices to sheafify with respect to $Y$.
	
	The right hand side is 
	\[ Y\mapsto H^n_{\et}(X\times Y,\bOx).\]
	We would like to argue that theses two presheaves are isomorphic on affinoid perfectoid $Y$ by Proposition~\ref{p:rigid-approxi-for-bOx} in the case of $U=X\times Y$. To see that this applies, it remains to prove:
\begin{Proposition}\label{l:rigid-approx-of-aff-spaces}
	Let $(K,K^+)$ be a non-archimedean field over $\Z_p$ and let $Y\to Y_0$ be a morphism of affinoid adic spaces over $(K,K^+)$. Assume that $Y$ is uniform and that for any affinoid smooth morphism $Y_1\to Y_0$ of topologically finite type also $Y_1$ is uniform. For example, this is satisfied when $Y$ and $Y_0$ are sousperfectoid. Let $(Y\to Y_i)_{i\in I}$ be the cofiltered inverse system of all morphisms of adic spaces from $Y$ into affinoid adic spaces $Y_i$ that are smooth of topologically finite type over $Y_0$. Then
	\[Y\approx\varprojlim_{Y\to Y_i} Y_i.\]
	In particular, if $K$ is perfectoid, then for any abelian sheaf $F$ on $\mathrm{SmRig}_{K,\et}$ and any $n\geq 0$, we have
	\[ H^n_{\et}(Y,F^{\diamond})=\varinjlim_i H^n_{\et}(Y_i,F). \]
	All of this remains true if we instead take the $Y_i$ to be open subspaces of unit balls over $Y_0$.
\end{Proposition}

\begin{proof}
	That the assumptions are satisfied when $Y$ and $Y_0$ are sousperfectoid follows from \cite[Propositions 6.3.3 and 6.3.4]{ScholzeBerkeleyLectureNotes}.

	Let $0\neq \varpi\in K^+$ be a pseudo-uniformiser.  As the very first step, we consider a different inverse system that is not yet smooth:
	Write $Y=\Spa(S,S^+)$ and $Y_0=\Spa(S_0,S_0^+)$ and let $\mathcal J$ be the partially ordered set of finite subsets of $S^+$. For $J\in \mathcal J$, let $S_J$ be the image of 
	\[\phi_J:S_0\langle X_j|j\in J\rangle\to S,\quad X_j\mapsto j,\]
	and let $S^+_J$ be the integral closure of the image $S_{J,0}$ of $S_0^+\langle X_j|j\in J\rangle$ in $S_J$. 
	Then $S_J\subseteq S$, and the $Z_J:=\Spa(S_J,S^+_J)$ form a cofiltered inverse system of adic spaces of topologically finite type over $Y_0$ such that $\varinjlim_{J\in \mathcal J} S^+_J\to S^+$ is an isomorphism by construction: In fact, already $\varinjlim_{J\in \mathcal J} S_{J,0}\to S^+$ is an isomorphism. We thus have
	$ Y\approx\varprojlim_{J\in \mathcal J} Z_J$
	where we use \cite[Proposition~2.4.2]{ScholzeWeinstein} to see the required statement about the underlying topological spaces. Here we use that $S$ is uniform, so $S^+$ has the $\varpi$-adic topology.
	
	Passing to the inverse system $ I$ in the Lemma, we note that any morphism $Y\to Y_i$ to a smooth affinoid adic space over $Y_0$ factors through some $Z_J$. We thus get a well-defined map
	\begin{equation}\label{eq:approx-Y-by-various-rigid-systems}
		\varinjlim_{i\in  I} \O^+(Y_i)\to \varinjlim_{J\in\mathcal J} \O^+(Z_J)=\varinjlim_{J\in \mathcal J} S_{J,0}
	\end{equation}
	and it suffices to prove that this becomes an isomorphism after $\varpi$-adic completion.
	
	We first note that the map is surjective: this is because any $Z_J$ has by its definition via $\phi_J$ a closed immersion into a closed ball $Z_J\hookrightarrow\B^{J}$, and $\B^{J}$ is smooth and thus appears as one of the $Y_i$ on the left hand side. Thus $S_{J,0}$ is in the image.
	
	To see that it is in fact an isomorphism after $\varpi$-adic completion, let $Y\to Z=\Spa(R,R^+)$ be any morphism into a smooth affinoid over $Y_0$. Let $Z_0\subseteq Z$ be the closure of the image, i.e.\ the closed subspace cut out by the kernel $N\subseteq R$ of the corresponding map $R\to S$.
	\begin{claim}\label{cl:closed-is-proj-limit-of-open-nbhds}
		We have $Z_0\approx\varprojlim_{Z_0\subseteq U\subseteq Z}U$ where $U$ ranges through the rational open neighbourhoods of $Z_0$ in $Z$.
	\end{claim}
	\begin{proof}
		Clearly $\varinjlim \O(U)\to \O(Z_0)$ is even surjective, so it suffices to check the condition on topological spaces:
		Both sides are subspaces of $|Z|$, so the map is necessarily a homeomorphism onto its image. It is also surjective: Let $x\in |Z|\backslash |Z_0|$, then there is $f\in N$ such that $|f(x)|\neq 0$. Since $\varpi$ is topologically nilpotent and $Z$ is quasi-compact, it follows that there is $k$ such that $|f(x)|> |\varpi^k|$ on $Z$. Thus $|f|\leq |\varpi^k|$ defines a rational open neighbourhood of $Z_0$ that does not contain $x$.
	\end{proof}
	Lemma~\ref{l:dense-image-implies-isomorphism-on-O^+/varpi} now implies that $\varinjlim_{Z_0\subseteq U\subseteq Z} \O^+(U)\to \O^+(Z_0)$ becomes an isomorphism mod $\varpi^k$. It follows that both sides of \eqref{eq:approx-Y-by-various-rigid-systems} agree mod $\varpi^k$ with 
	\[ \varinjlim_{Y\to U\subseteq \B^n_{Y_0}}\O^+(U)/\varpi^k,\]
	where the index category consists of morphisms from $Y$ into rational open subspaces $U\subseteq \B^n_{Y_0}$ of rigid polydiscs. This proves the first part of Lemma~\ref{l:rigid-approx-of-aff-spaces}.
	
	The part about cohomology now follows from Lemma~\ref{l:cohom-of-limit-topos}. Alternatively, we could follow the argument in \cite[Proposition~14.9]{etale-cohomology-of-diamonds}, or in \cite[Lemma~3.16, Corollary~3.17]{Scholze_p-adicHodgeForRigid}.
\end{proof}
\cref{l:rigid-approx-of-aff-spaces} applied to $Y_0=\Spa(K)$ thus finishes the proof of \cref{c:Lemma-B}.
\end{proof}
A  similar but much easier argument as in \cref{c:Lemma-B} completes step (E) of \cref{l:describing-master-diagram-2}:
\begin{Corollary}\label{c:base-change-for-etale-sheaves}
	For any $n,N\in \N$, the map from \eqref{eq:base-change-from-rigid-to-diamond} for $\mathcal F=\Z/N\Z$ is an isomorphism
	\[ (R^n\pi_{\et\ast}\Z/N\Z)^\diamondsuit\isomarrow R^n\pi_{\et\ast}^{\diamondsuit}\Z/N\Z.\]
\end{Corollary}
\begin{proof}
	Arguing as in the last proof, we see that the first term is the sheafification of
	\[Y\mapsto\varinjlim_{Y\to Z} H^n_{\et}(X\times Z,\Z/N),\]
	whereas the second term is the sheafification of $Y\mapsto H^n_{\et}(X\times Y,\Z/N)$. The two presheaves agree by \cite[Proposition~14.9]{etale-cohomology-of-diamonds} which applies by \cref{l:rigid-approx-of-aff-spaces}.
\end{proof}

In summary, we have in this section completed those parts of  Lemma~\ref{l:describing-master-diagram-1}.(B) and Lemma~\ref{l:describing-master-diagram-2}.(E) that concern the comparison of the bottom two rows.

Before we continue, let us record a variant of \Cref{p:rigid-approxi-for-bOx} which is useful for applications and which follows immediately from the above technical results:
\begin{Corollary}\label{c:tilde-limit-b0x}
	Let $Y$ be a perfectoid space over $K$ and let $(Y_i)_{i\in I}$ be a cofiltered inverse system of qcqs adic spaces over $K$ such that $Y\sim\varprojlim Y_i$ and such that every object of $Y_{i,\et}$ is sheafy for all $i\in I$. Assume that there is $i\in I$ and a cover of $Y_i$ by affinoid open subspaces $U_i\subseteq  Y_i$ such that $U:=Y\times_{Y_i}U_i$ is affinoid perfectoid and $U_j:=Y_j\times_{Y_i}U_i$ is affinoid for any $j\geq i$ and such that $U\approx \varprojlim_{j\geq i} U_j$. Then for any $n\geq 0$, the natural map
	\[ \varinjlim H^n_{\et}(Y_i,F)\to H^n_{\et}(Y,F)\]
	is an isomorphism
	for $F=\bOx$ or $F=\O^+/\varpi$  for any $0\neq \varpi\in \m$.
\end{Corollary}
\begin{proof}
	By \Cref{l:part-2-for-fet-on-Y} and \Cref{l:dense-image-implies-isomorphism-on-O^+/varpi}, we have for any  $V\in U_{i,\stdet}$ that \[\varinjlim_{j\geq i} \O^+(Y_j\times_{Y_i}V)/\varpi= \O^+(Y\times_{Y_i}V)/\varpi,\]
	and similarly for $\bOx$. After sheafification on $Y_{\etqcqs}=\twolim_iY_{i,\etqcqs}$, this shows that the assumptions of \Cref{l:cohom-of-limit-topos} are satisfied. This implies the Corollary.
\end{proof}
\section{Cohomology of products of rigid with perfectoid spaces}
The main aim of this section is to complete the proofs of \cref{l:describing-master-diagram-1} and  \cref{l:describing-master-diagram-2}. We start with cohomological computations for the sheaf $\O$, for which we also prove the results described in the first part of \cref{s:diamantine-comparison-explained}, namely \cref{p:HT-ses-vs} and \cref{p:HTses-of-R^1pi-for-O}.

Throughout this section, $K$ is a non-archimedean field extension of $\Q_p$.
\subsection{Cohomology of $\O$}
We begin with a few general lemmas on the cohomology of ``mixed tensor products'' of rigid with perfectoid spaces, which are essentially coherent base-change results. For these we crucially use Kiehl's Theorem that $R\Gamma(X,\O)$ is perfect for any proper rigid space  $X$.
\begin{Lemma}\label{l:cohomology-of-O-for-mixed-spaces}
	Let $X$ be an affinoid rigid space over $K$ and let $Y$ be an affinoid perfectoid space. Then for $n> 0$, we have 
	$H^n_{\et}(X\times Y,\O)= 0$.
\end{Lemma}
\begin{proof}
	This is true in much greater generality by an application of \cite[Theorem~8.2.22(c)]{KedlayaLiu-rel-p-p-adic-Hodge-I}. This applies here because \'etale maps that factor into rational embeddings and finite \'etale maps form a basis for $(X\times Y)_{\et}$ by Proposition~\cite[Proposition~11.31]{etale-cohomology-of-diamonds}.
\end{proof}

\begin{Proposition}\label{p:affinoid-perfectoid-primitive-comparison}
	Let $X$ be a smooth proper rigid space over $K$.
	\begin{enumerate}
		\item Let $Y$ be any smooth affinoid rigid space. Then
		\[ H^n_{\et}(X\times Y,\O)=H^n_{\et}(X,\O)\otimes_{K}\O(Y).\]
		\item Let $Y$ be an affinoid perfectoid space over $K$. Then there are natural isomorphisms: 
		\begin{alignat*}{3}
	&(i)&H^n_{\et}(X\times Y,\O)&\,=\,&&H^n_{\et}(X,\O)\otimes_K\O(Y),\\
	&(ii)&\quad H^n_v(X\times Y,\O)&\,=\,&& H^n_{v}(X,\O)\otimes_{K}\O(Y),\\
	&(iii)&\quad H^n_v(X\times Y,\O^+/p^k)&\,\aeq\,&& H^n_{v}(X,\O^+/p^k)\otimes_{K^+}\O^+/p^k(Y).
	\end{alignat*}

	In particular, $R^n\pi^\diamond_{\et\ast}\O=H^n_{\et}(X,\O)\otimes_K \O$ and $R^n\pi^\diamond_{v\ast}\O=H^n_{v}(X,\O)\otimes_{K} \O$.
	\item Let $Y$ be an affinoid perfectoid space over $K$ and assume that $K$ is algebraically closed. Then the following natural map is an almost isomorphism:
	\[H^n_{\et}(X,\Z_p)\otimes_{\Z_p} \O^+(Y)\to H^n_v(X\times Y,\O^+).\]
	\end{enumerate}
\end{Proposition}
We remark that these statements are all easier special cases of a much more general adic version of Grothendieck's ``cohomology and base-change'' which will be proved in the sequel \cite[Theorem 3.18]{heuer-relative-HT}. For example, smoothness is not necessary for \Cref{p:affinoid-perfectoid-primitive-comparison}.1.
\begin{proof}
	We start with  part 2.(ii): 	Since $X$ is quasi-compact separated, we can choose a finite cover $\mathcal U$ of $X$ by affinoids $U_i$ with affinoid intersections that are \'etale over a torus, and thus admit toric pro-finite-\'etale covers $\wt U_i\to U_i$. Then any fibre product of the $\wt U_i$ over $X$ is affinoid perfectoid. Consequently, the cohomology $H^n_v(X,\O^+)$ is almost computed by the \cH\ complex $\check{C}^\bullet(\wt {\mathcal U},\O^+)$ where $\wt{\mathcal U}$ is the pro-\'etale cover of $X$ by the $\wt U_i$.
	
	Each $H^n_v(X,\O^+)$ has bounded $p$-torsion: This follows from the fact that the sheaves in \cref{p:Scholze-v-cohom-O} are coherent, so $H^n_v(X,\O)$ is finite dimensional. By an application of \Cref{l:completion-of-Banach-cpx}.1 in the appendix, this implies that $\check{C}^\bullet_v(\wt {\mathcal U},\O^+)$ is a complex of $p$-torsionfree $p$-complete $K^+$-modules whose cohomology has bounded $p$-torsion.
	
	\medskip
	
	We now add the factor  $Y$: Clearly the $\wt U_i\times Y$ form a cover $\wt{\mathcal U}\times Y$ of $X\times Y$ such that all spaces appearing in the \cH\ nerve are still affinoid perfectoid. Let $\O^+(Y)=S^+$, then the fact that $\check{C}^\bullet_v(\wt {\mathcal U},\O^+)$ has cohomology of bounded torsion implies by \Cref{l:completion-of-Banach-cpx}.2 that
	\begin{alignat*}{2}
	H^n_v(X\times Y,\O^+)&\aeq&& H^n(\check{C}^\bullet(\wt{\mathcal U}\times Y,\O^+))\aeq H^n(\check{C}^\bullet(\wt {\mathcal U},\O^+)\hotimes_{K^+} S^+)\\
	 &\aeq&&  H^n(\check{C}^\bullet(\wt {\mathcal U},\O^+))\hotimes_{K^+} S^+\\
	&\aeq&& H^n_v(X,\O^+)\otimes_{K^+} S^+.
\end{alignat*}
	After inverting $p$, this gives the desired equality for 2.(ii). 
	
	Part 2.(iii) also follows from the displayed equation by comparing the long exact sequences of $0\to \O^+\to \O^+\to \O^+/p^k\to 0$ for $X\times Y$ and $X$, using the 5-Lemma.
	
	If $K$ is algebraically closed, we also deduce part 3 using the Primitive Comparison Theorem, \cite[Theorem~5.1]{Scholze_p-adicHodgeForRigid}.
	
	Part 2.(i) follows by a similar argument using instead the cover $\mathcal U$: By Lemma~\ref{l:cohomology-of-O-for-mixed-spaces}, the group $\cH^n(\mathcal U,\O\hotimes S)$ computes $H^n_{\et}(X\times Y,\O)$. Since each $\check{H}^n( {\mathcal U},\O)$ is finite, we can now again apply \Cref{l:completion-of-Banach-cpx}.2 to the complex of $K^+$-modules $\check{C}^\bullet( {\mathcal U},\O^+)$ to see that:
	\[ H^n_{\et}(X\times Y,\O)= H^n(\check{C}^\bullet(\mathcal U,\O^+)\hotimes_{K^+} S^+)\tf= \check{H}^n( {\mathcal U},\O^+)\hotimes_{K^+} S^+\tf= H^n_{\et}(X,\O)\otimes S.\]
	
	Part 1 can be seen similarly:
By Tate acyclicity, $\cH^n(\mathcal U\times Y,\O)$ computes $H^n(X\times Y,\O)$. For any affinoid $U\subseteq X$, the map $\O^+(U)\hotimes_{K^+}\O^+(Y)\to\O^+(U\times Y)$ has bounded $p$-torsion cokernel since $U\times Y$ is uniform. Hence in the composition
	\[ H^n(\check{C}^\bullet(\mathcal U,\O^+))\hotimes_{K^+}\O^+(Y)\to H^n(\check{C}^\bullet(\mathcal U,\O^+)\hotimes_{K^+}\O^+(Y))\to H^n(\check{C}^\bullet(\mathcal U\times Y,\O^+)),\]
	the second map becomes an isomorphism after inverting $p$, while the first map is an isomorphism by \Cref{l:completion-of-Banach-cpx}. After inverting $p$, this gives the desired statement.
\end{proof}

Second, we need the following result of Scholze:
\begin{Proposition}[{\cite[Proposition~3.23]{Scholze2012Survey},
		\cite[2.24--2.25]{heuer-v_lb_rigid}}]\label{p:Scholze-v-cohom-O}
	Let $X$ be any smooth rigid space and let $\nu:X_{v}\to X_{\et}$ be the natural morphism of sites. Then $R^n\nu_{\ast}\O=\wedge^n\wtOm^1_X$.
\end{Proposition}

\begin{proof}[Proof of \cref{p:HT-ses-vs}]
	Using \cref{p:Scholze-v-cohom-O}, we see that the 5-term exact sequence of the Leray sequence for the morphism $\nu$ is of the form
	\begin{equation}\label{eq:partial_X} 
		0\to  H^1_\et(X,\O)\to  H^1_v(X,\O)\to H^0(X,\wtOm^1_X)\xrightarrow{\partial_X} H^2_{\et}(X,\O)\xrightarrow{j_X} H^2_{v}(X,\O).
	\end{equation}
	 If $K$ is algebraically closed, it follows from the degeneration of the Hodge--Tate spectral sequence  \cite[Theorem 1.7.(ii)]{BMS} that $\partial_X=0$. This implies that the fourth map $j_X$ is injective.
	
	The general case follows from this: It suffices to prove that $\partial_X=0$, or equivalently that $j_X$ is injective. Let $C$ be the completion of an algebraic closure of $K$.
	 By  \cref{p:affinoid-perfectoid-primitive-comparison} for $Y=\Spa(C)$, the base-change of $j_X$ along $K\to C$ admits an identification $j_X\otimes_K C=j_{X_C}$. This is injective by the algebraically closed case, hence $j_X$ is injective.
\end{proof}
We now move on to the relative Hodge--Tate sequence:
\begin{proof}[Proof of \cref{p:HTses-of-R^1pi-for-O}]
		The first part follows from comparing Proposition~\ref{p:affinoid-perfectoid-primitive-comparison}.1 and 2.
		
		To see the second part, we tensor the Hodge--Tate sequence for $X$ from \cref{p:HT-ses-vs} with $\O$ and see from Proposition~\ref{p:affinoid-perfectoid-primitive-comparison}.2.(i) and (ii) for $i=1$ that we obtain identifications
		\[
		\begin{tikzcd}
			0 \arrow[r] & R^1\pi^\diamond_{\et\ast}\O \arrow[r]                    & R^1\pi^\diamond_{v\ast}\O                 &   \\
			0 \arrow[r] & {H^1_{\et}(X,\O)\otimes \O} \arrow[u,"\sim"labelrotate] \arrow[r] & {H^1_v(X,\O)\otimes \O} \arrow[u,"\sim"labelrotate] \arrow[r] & {H^0(X,\wtOm^1(X))\otimes \O}  \arrow[r] & 0.\qedhere
		\end{tikzcd}\]
	\end{proof}

We can now also address Lemma~\ref{l:describing-master-diagram-2}.(F):

\begin{Proposition}\label{p:HTses-of-R^1pi-for-O-s4}
	The map $(R^2\pi_{\et\ast}\O)^\diamondsuit\to R^2\pi^\diamondsuit_{v\ast}\O$ is injective.
\end{Proposition}
\begin{proof}
	In the notation of the sequence \eqref{eq:partial_X} above,
	 Proposition~\ref{p:affinoid-perfectoid-primitive-comparison}.2 for $i=2$ identifies this map with
	$j_X\otimes \O:H^2_{\et}(X,\O)\otimes \O\hookrightarrow H^2_{v}(X,\O)\otimes \O$.
	This is injective by \Cref{p:HTses-of-R^1pi-for-O}.
\end{proof}

From the case of $i=0$ of Proposition~\ref{p:affinoid-perfectoid-primitive-comparison}, we will moreover deduce part (D) of Lemma~\ref{l:describing-master-diagram-2} (see the end of this subsection). For this we will use the following consequence:
\begin{Corollary}\label{c:pi_astO^+}
	Suppose that $X$ is geometrically connected. Then \[(\pi_{\et\ast}\O)^\diamondsuit=\O=\pi^{\diamondsuit}_{\et \ast}\O=\pi^{\diamondsuit}_{v \ast}\O.\] The analogous statements hold for $\Oone$, $\O^\times$ and $\O^{+a}/p^k$.
\end{Corollary}
\begin{proof}
	The first part follows from \cref{p:affinoid-perfectoid-primitive-comparison}.1-2.(i): Here we use that $H^0(X,\O)=K$ since $X$ is geometrically connected. The cases of $\O^\times$ and $\Oone$ follow as these are subsheaves of $\O$.  
	
	For $\O^{+a}/p^k$, we first recall that we have $(\pi_{\et\ast}(\O^+/p^k))^\diamondsuit=\pi^{\diamondsuit}_{\et \ast}(\O^+/p^k)$ by \Cref{c:Lemma-B}. Second, we have $\pi^{\diamondsuit}_{\et \ast}(\O^+/p^k)=\pi^{\diamondsuit}_{v \ast}(\O^+/p^k)$  by \cref{l:bOx-et-vs-v-on-XxY}. Finally, it follows from \cref{p:affinoid-perfectoid-primitive-comparison}.2.(iii) that $\pi^{\diamondsuit}_{v \ast}(\O^+/p^k)=H^0_v(X,\O^+/p^k)\otimes \O^+/p^k$. It thus remains to see that $H^0_v(X,\O^+/p^k)\aeq K^+/p^k$. For this we can use \Cref{p:rigid-approxi-for-bOx} to reduce to the case that $K$ is algebraically closed, where the statement follows from \cref{p:affinoid-perfectoid-primitive-comparison}.3.
\end{proof}

We can also deduce a  version of the Primitive Comparison Theorem relatively over $Y$:
\begin{Corollary}\label{c:PCT-for-XxY}
	Assume that $K$ is algebraically closed.
	Let $X$ be a smooth proper rigid space over $K$ and let $Y$ be affinoid perfectoid over $K$. Then the natural map
	\[H^n_{v}(X,\Z/p^k)\otimes \O^+(Y)/p^k\to H^n_v(X\times Y,\O^+/p^k)\]
	is an almost isomorphism for all $n\geq 0$. In particular, the natural map
	\[H^n_{v}(X,\F_p)\otimes \O^{\flat+}(Y)\to H^n_v(X\times Y,\O^{\flat+})\]
	is an almost isomorphism for all $n\geq 0$, compatible with Frobenius actions on both sides.
\end{Corollary}
\begin{proof}
	The first part follows from Proposition~\ref{p:affinoid-perfectoid-primitive-comparison}.3, using the  sequence $\O^+\to \O^+\to \O^+/p^k$ and the fact that $H^n_{\et}(X,\Z/p^k)=H^n_v(X,\Z/p^k)$ by \cite[Propositions 14.7 and 14.8]{etale-cohomology-of-diamonds}.
	The second part follows from the case of $k=1$ in the inverse limit over Frobenius.
\end{proof}
\begin{Proposition}[K\"unneth formula]\label{p:Kunneth}
	Let $X$ be a smooth proper rigid space and let $Y$ be affinoid perfectoid. Then there is a natural isomorphism for all $n\geq 0$
	\[ H^n_{v}(X\times Y,\F_p)=\Big(H^{n-1}_v(X,\F_p)\otimes H^{1}_v(Y,\F_p)\Big)\oplus \Big(H^{n}_v(X,\F_p)\otimes H^{0}_v(Y,\F_p)\Big).\]
\end{Proposition}
\begin{proof}
	We consider the $v$-cohomological long exact sequence  for the  Artin--Schreier sequence \[0\to \F_p\to \O^\flat\xrightarrow{\mathrm{AS}} \O^\flat\to 0\] on $X\times Y$. By Corollary~\ref{c:PCT-for-XxY}.2, this yields a long exact sequence
	\[ \dots\xrightarrow{\mathrm{AS}} H^{n-1}(X,\F_p)\otimes\O^{\flat}(Y)\to H^{n}_v(X\times Y,\F_p)\to H^{n}(X,\F_p)\otimes\O^{\flat}(Y)\xrightarrow{\mathrm{AS}} \dots \]
	Since $H^n_v(Y,\O)=0$ for $n\geq 1$, we have $H^1_v(Y,\F_p)=\coker(\mathrm{AS}:\O^{\flat}(Y)\to \O^{\flat}(Y))$ and $H^n_v(Y,\F_p)=0$ for $n\geq 2$. It follows that we can rewrite the above as a natural extension
	\[ 0\to H^{n-1}_v(X,\F_p)\otimes H^{1}_v(Y,\F_p)\to H^n_{v}(X\times Y,\F_p)\to H^{n}_v(X,\F_p)\otimes H^{0}_v(Y,\F_p)\to 0.\]
	Recall that $\pi_0(Y)$ is always a profinite space \cite[Tag 0906]{StacksProject}. By comparing to the case that $Y=\underline{\pi_0(Y)}$ is strictly totally disconnected, in which case $H^1_v(Y,\F_p)=0$, we see that pullback along $X\times Y\to X\times \underline{\pi_0(Y)}$ defines a natural splitting of the last map.
\end{proof}
We use this to complete the second part of Lemma~\ref{l:describing-master-diagram-1}.(E):
\begin{Corollary}\label{c:R^npi_Fp-comparison}
	For any $N\in \N$ and $n\geq 0$, we have a natural isomorphism
	\[ R^n\pi_{\et\ast}^\diamond\Z/N\Z=R^n\pi_{v\ast}^\diamond\Z/N\Z.\]
	If $K$ is algebraically closed, this is isomorphic to $\underline{H^n_{\et}(X,\Z/N\Z)}$, the locally constant sheaf on $\Perf_{K,v}$ associated to the group $H^n_{\et}(X,\Z/N\Z)$.
\end{Corollary}
\begin{proof}
	For $N$ coprime to $p$, this follows from general base-change results for the diagram
	\[ \begin{tikzcd}
		X^\diamond_v \arrow[r, "\pi_v^\diamondsuit"] \arrow[d] & \Perf_{K,v} \arrow[d, "\nu"] \\
		X^\diamond_{\et} \arrow[r, "\pi^\diamondsuit_{\et}"] & \Perf_{K,\et},
	\end{tikzcd}\]
	namely by 
	\cite[Theorem 16.1.(iii) and Proposition~16.6]{etale-cohomology-of-diamonds}, the base-change morphism
	\[\nu^\ast R^n\pi_{\et\ast}^\diamond\Z/N\Z\to R^n\pi_{v\ast}^\diamond\Z/N\Z\]
	is an isomorphism, and $\nu_{\ast}\nu^\ast F=F$ for any sheaf on $\Perf_{K,\et}$ by \cite[Proposition 14.7]{etale-cohomology-of-diamonds}. 
	
	The last sentence of the Corollary is clear when $K^+=\O_K$ since any sheaf on $\Spa(K,\O_K)_\et$ is constant. The general case follows from this: Let $j:\Spa(K,\O_K)\hookrightarrow \Spa(K,K^+)$ be the natural open immersion, then it follows from \cite[Proposition 8.1.2.(ii)]{huber2013etale} that \[R^n\pi_{\et\ast}\Z/N\Z=j_{\ast}j^{\ast}R^n\pi_{\et\ast}\Z/N\Z=j_{\ast}\underline{H^n_{\et}(X,\Z/N\Z)}=\underline{H^n_{\et}(X,\Z/N\Z)}.\]
	
	For $N$ a power of $p$, we can reduce by induction to the case of $N=p$. Then $R^n\pi_{v\ast}^\diamond\F_p$ is the $v$-sheafification of 
	\[ Y\mapsto H^n_v(X\times Y,\F_p).\]
	By Proposition~\ref{p:Kunneth}, this is the locally constant sheaf of $H^n_{\et}(X,\F_p)=H^n_v(X,\F_p)$.
\end{proof}
At this point, we can complete the proof of Lemma~\ref{l:describing-master-diagram-2}:
\begin{Corollary}\label{c:part-E-of-master-lemma-2}
	For any $n\geq 0$, the following morphisms are isomorphisms:
	\[(R^n\pi_{\et\ast}\mu_{p^\infty})^\diamond\to R^n\pi_{\et\ast}^\diamond\mu_{p^\infty}\to R^n\pi_{v\ast}^\diamond\mu_{p^\infty}\]
\end{Corollary}
\begin{proof}
	By quasi-compactness, it suffices to prove this for $\mu_{p^\infty}$ replaced by $\mu_{p^m}$. We can check the statement locally on $\Perf_{K,\et}$, and may therefore assume that $K$ contains  $\mu_{p^m}(\overline{K})$. Then $\mu_{p^m}\cong \underline{\Z/p^m\Z}$ and the statement follows from \cref{c:base-change-for-etale-sheaves} and \cref{c:R^npi_Fp-comparison}.
\end{proof}
\begin{proof}[Proof of Lemma~\ref{l:describing-master-diagram-2}]
	For Part  (D), we use that by \cref{c:pi_astO^+}, the left-exact sequence
	\[ 0\to \pi^{\diamondsuit}_{\et \ast}\mu_{p^\infty}\to \pi^{\diamondsuit}_{\et \ast}\Oone\to \pi^{\diamondsuit}_{\et \ast}\O \to 0\]
	 is also right-exact.
	The analogous statements hold for $\pi^{\diamondsuit}_{v \ast}\O$ and $(\pi_{\et \ast}\O)^{\diamondsuit}$. This implies Lemma~\ref{l:describing-master-diagram-2}.(D).
	Part (E) is \cref{{c:part-E-of-master-lemma-2}}, (F) is \cref{p:HTses-of-R^1pi-for-O-s4}. 
\end{proof}

\subsection{Cohomology of $\bOx$}
We now move on to proving the remaining parts of Lemma~\ref{l:describing-master-diagram-1} concerning the sheaf $\bOx$. For some relevant background information on this sheaf, the interested reader might find it helpful  to look at \cite[\S2.3-\S2.4]{heuer-v_lb_rigid} on which some arguments in the following are based. 
We begin with some preparations.

\begin{Lemma}\label{l:bOx-locally-constant}
	Let $X$ be a rigid space over an algebraically closed field $K$. Then evaluation at points in $X(K)$ induces a unique injective map fitting into the commutative diagram:
	\[
	\begin{tikzcd}
		\O^\times(X) \arrow[d] \arrow[r] & {\Map_{\cts}(X(K),K^\times)} \arrow[d] \\
		\bOx(X) \arrow[r,hook,dotted]                & {\Map_{\lc}(X(K),K^\times/(1+\m)).}    
	\end{tikzcd}.\]
\end{Lemma}
\begin{proof}
	The first arrow is given by interpreting $f\in \O^\times(X)$ as a morphism $X\to \G_m$ and evaluating on $K$-points. By the Maximum Modulus Principle, this sends $f\in \O^\times(X)$ into $\Map_{\cts}(X(K),1+\m)$ if and only if $f\in \Oone(X)$.
	It now suffices to construct the bottom map for affinoid $X$, where $\bOx(X)=\O^\times\tf(X)/\Oone(X)$ by \cite[Lemma~2.19]{heuer-v_lb_rigid}
\end{proof}
\begin{Lemma}\label{l:H^1(bOx)-on-XxY-insensitive-to-sec-factor}
	Let $X$ be a smooth proper rigid space over $K$ that is geometrically connected.
	\begin{enumerate}
		\item For $Y$ any smooth rigid space over $K$, we have $\bOx(X\times Y)=\bOx(Y)$.
		\item For $Y$ any perfectoid space over $K$, we have $\bOx(X\times Y)=\bOx(Y)$.
	\end{enumerate}
	In particular, we have $(\pi_{\et\ast}\bOx)^\diamondsuit=\pi^\diamondsuit_{\et\ast}\bOx=\pi^\diamondsuit_{v\ast}\bOx=\bOx$.
\end{Lemma}
\begin{proof}
	We first explain how to deduce part 2 from part 1: The statement is local on $Y$, so we can assume that $Y$ is affinoid perfectoid. By \Cref{l:rigid-approx-of-aff-spaces},  we can then find an inverse system of affinoid smooth rigid spaces $(Y_i)_{i\in I}$ over $K$ such that $Y\approx \varprojlim Y_i$. Assuming part~1, we then have by \Cref{p:rigid-approxi-for-bOx}:
	\[ \bOx(X\times Y)=\varinjlim_{i\in I}\bOx(X\times Y_i)=\varinjlim_{i\in I}\bOx(Y_i)=\bOx(Y).\]
	
	Next, let us explain the last sentence of the Lemma: Recall that we had already seen in \Cref{c:Lemma-B} that $(\pi_{\et\ast}\bOx)^\diamondsuit=\pi^\diamondsuit_{\et\ast}\bOx$. That $\pi^\diamondsuit_{\et\ast}\bOx=\pi^\diamondsuit_{v\ast}\bOx$ follows from \Cref{l:bOx-et-vs-v-on-XxY}. Part~2 implies $\pi^\diamondsuit_{v\ast}\bOx =\bOx$ immediately from the definition.
	
	It thus remains to prove part 1. For this, we can assume that $Y$ is affinoid and connected. Second, we can without loss of generality assume that $K$ is algebraically closed: Let $\overline{K}$ be an algebraic closure of $K$ and let $C$ be its completion. Then to deduce the general case from that over $C$, let $G:=\mathrm{Gal}(\overline{K}|K)$ and consider the $G$-torsor $X\times Y_C\to X\times Y$. Assuming part 1 for $C$, and using that $\bOx$ is a v-sheaf on the smooth rigid space $X\times Y$ by \Cref{l:bOx-et-vs-v-on-XxY} (applied with ``$Y$'' in \Cref{l:bOx-et-vs-v-on-XxY} being $\Spa(K)$), we then have
	\[\bOx(X\times Y)=\bOx(X\times Y_C)^G=\bOx(Y_C)^G=\bOx(Y).\]
	
	Now for  algebraically closed $K$, we start with $Y=\Spa(K)$. In this case, using that $X$ is connected, we compare to the universal pro-\'etale cover $\wt X\to X$ of \cite[\S 4]{heuer-v_lb_rigid}: Using the exponential sequence $0\to \O\to \O^\times\tf\to \bOx\to 1$ from \cite[Lemma 2.18]{heuer-v_lb_rigid}, we have a commutative diagram
	\[
	\begin{tikzcd}
		\O^\times\tf(\wt X) \arrow[r]           & \bOx(\wt X) \arrow[r]             & {H^1(\wt X,\O)=0}     \\
		\O^\times\tf(X) \arrow[r] \arrow[u] & \bOx(X) \arrow[u, hook] \arrow[r] & {H^1_\et(X,\O)} \arrow[u]
	\end{tikzcd}\]
	in which by \cite[Proposition~3.10]{heuer-v_lb_rigid}, the top row can be identified with the sequence
	\[K^\times\tf\to K^\times/(1+\m)\to 1.\]
	In particular, the first vertical arrow is surjective. The second vertical arrow is injective since $\wt X\to X$ is a Galois cover. This shows that $\bOx(X)=\bOx(\wt X)=K^\times/(1+\m)$. In particular, the boundary map $\bOx(X)\to H^1_\et(X,\O)$ vanishes.
	
	We now move on to the case of general affinoid and connected smooth rigid spaces $Y$: We will show that the boundary map $\partial$ of the exponential sequence
	\[0\to	H^0(X\times Y,\O)\to H^0(X\times Y,\O^\times)\tf\to H^0(X\times Y,\bOx)\xrightarrow{\partial} H^1_{\et}(X\times Y,\O) \]
	vanishes as well. This implies the desired result: We already know from \Cref{c:pi_astO^+} that the first two terms identify with $\O(Y)$ and $\O^\times(Y)\tf$. Comparing to the same sequence for $X=\Spa(K)$, we see that their quotient is $\bOx(Y)$ because $H^1_\et(Y,\O)=0$.
	
	To see that $\partial=0$, we can without loss of generality assume that $(K,K^+)=(K,\O_K)$: Indeed, pullback along $\Spa(K,\O_K)\to \Spa(K,K^+)$ only changes the integral subrings, so the comparison map $H^1_{\et}(X\times Y,\O)\to H^1_{\et}(X\times Y\times_{\Spa(K,K^+)}\Spa(K,\O_K),\O)$ is an isomorphism.
	
	The idea is now to compare the boundary map of the exponential sequence for $X$ and $X\times Y$ via the pullback along $X\to X\times Y$ for points in $Y(K)$. This results in a commutative diagram:
	\[
	\begin{tikzcd}
		H^0(X\times Y,\bOx)\arrow[d] \arrow[r] & {H^1_{\et}(X\times Y,\O)} \arrow[d,hook]  \\
		{\Map(Y(K),\bOx(X))} \arrow[r]  & {\Map(Y(K),H^1_{\et}(X,\O))}
	\end{tikzcd}\]
	By Proposition~\ref{p:affinoid-perfectoid-primitive-comparison}.1, we know that
	\[H^1_\et(X\times Y,\O)=H^1_{\et}(X,\O)\otimes_{K}\O(Y).\]
	Second, since $H^1_{\et}(X,\O)$ is a finite dimensional $K$-vector space, we have 
	\[ \Map(Y(K),H^1_{\et}(X,\O))=H^1_{\et}(X,\O)\otimes_K\Map(Y(K),K).\]
	This shows that the right vertical map can be identified with $H^1(X,\O)\otimes_K-$ applied to the evaluation map $\O(Y)\to \Map(Y(K),K)$. This is injective since $Y$ is a reduced affinoid classical rigid space: Indeed, if $f\in \O(Y)$ satisfies $f(x)=0$ for all $x\in Y(K)$, then since $K$ is algebraically closed, its supremum norm is $0$, which implies $f=0$ when $Y$ is reduced \cite[\S6.2 Proposition 4.(iii)]{BGR}. Thus the right vertical map is injective. But the bottom map is $=0$ by the case of $Y=\Spa(K)$. Hence the top map vanishes, as we wanted to see.
\end{proof}
Finally, we turn to the remaining part of  Lemma~\ref{l:describing-master-diagram-1} (B), about the top morphism. We need to compare \'etale and $v$-cohomology of $\bOx$ on products of $X$ with perfectoid spaces:
\begin{Proposition}\label{p:bOx-cohom-et-vs-v}
	Let $\mathcal F=\O^{+a}/p$ and $n\in \N$; or $\mathcal F=\bOx$  and $n\in \{0,1\}$.
	\begin{enumerate}
		\item Let $X$ be a smooth qcqs rigid space and let $Y$ be an affinoid perfectoid space over $K$. Then
		\[H^n_{\et}(X\times Y,\mathcal F)=H^n_{v}(X\times Y,\mathcal F).\]
		\item Let $Y$ be any spatial diamond over $K$. Let $\wt X=\varprojlim_{i\in I} X_i$ be a diamond which is a limit of smooth qcqs rigid spaces over $K$ with finite \'etale transition maps. Then
		\[\textstyle\varinjlim_i H^n_{v}(X_i\times Y,\mathcal F)= H^n_{v}(\wt X\times Y,\mathcal F).\]
	\end{enumerate}
\end{Proposition}
Part 1 is proved in much greater generality in \cite[Proposition~2.14]{heuer-G-torsors-perfectoid-spaces}, but we will in the next section also need part 2, from which it is easy to deduce part 1 independently.
\begin{proof}
	The proof will be completed in several steps:
	\begin{step}
	We first observe that if $Y$ is any qcqs perfectoid space and  $S=\varprojlim S_i$ is a profinite space, then  we can apply \Cref{c:tilde-limit-b0x} to $S\times Y\sim \varprojlim S_i\times Y$ to see that
	\[H^n_{\et}(S\times Y,\mathcal F)=\varinjlim H^n_{\et}(S_i\times Y,\mathcal F)=\varinjlim\Map(S_i,H^n_{\et}(Y,\mathcal F)).\]
	\end{step}
	\begin{step}
		When $Y$ is affinoid perfectoid, then we have
	\[\varinjlim H^n_{\et}(X_i\times Y,\mathcal F)= H^n_{v}(\wt X\times Y,\mathcal F).\]
	\end{step}
	\begin{proof}
	Choose any element $0\in I$.
	By a \cH-argument, it suffices to prove the statement after replacing $X_0$ by a qcqs open $U$ that admits a perfectoid (say, toric) cover $X_{\infty,0}\sim\varprojlim X_{j,0}\to X_0$ with pro-finite-\'etale Galois group $G=\varprojlim G_j$.
	
	Let $X_{j,i}:=X_{j,0}\times_{X_0}X_i$ and to simplify notation let $Z_{j,i}:=X_{j,i} \times Y$. Furthermore, let
	\[ \wt Z_j:=\wt X\times_{X_0}Z_{j,0}=\textstyle\varprojlim_i Z_{j,i}.\]
	In this notation, our space $\wt X\times Y$ is $\wt Z_{0}$. In summary, we have a commutative diagram
	\[\begin{tikzcd}
		\wt Z_\infty \arrow[d] \arrow[r] & Z_{\infty,0} \arrow[d] \\
		\wt Z_0  \arrow[r] & Z_{0,0}
	\end{tikzcd}\]
	in which the left map is a pro-finite-\'etale $G$-torsor under a perfectoid space, and the top morphism is a pro-finite-\'etale morphisms of perfectoid spaces.

	Since $\wt Z_\infty$ is perfectoid  we have 
	$H^n_{v}(\wt Z_\infty,\O^+/p)\aeq H^n_{\et}(\wt Z_\infty,\O^+/p)$ because $\O^+/p$ is almost acyclic on affinoid perfectoid spaces. The same holds for $\bOx$ when $n\in \{0,1\}$: This follows from the exponential sequence of \cite[Lemma~2.18]{heuer-v_lb_rigid}, using that $H^1_{\et}(T,\O^\times)=H^1_{v}(T,\O^\times)$ for any perfectoid space $T$ (see \cite[Theorem 3.5.8]{KedlayaLiu-rel-p-p-adic-Hodge-I} or \cite[Lemma 17.1.8]{ScholzeBerkeleyLectureNotes}). This is where we use the assumption that $n\in \{0,1\}$ when dealing with $\bOx$.
	
	We endow $H^n_{\et}(\wt Z_\infty,\O^+/p)$ with the discrete topology.
	By Step 1, we then have for $k\in \N$: \[H^n_{\et}(\wt Z_\infty\times G^k,\O^+/p)=\Map_{\cts}(G^k,H^n_{\et}(\wt Z_\infty,\O^+/p)).\]
	It follows that the \cH-to-sheaf spectral sequence of $\wt Z_\infty\to \wt Z_0$ is a Cartan--Leray spectral sequence in the almost category
	\[ H^n_{\cts}(G,H^m_{\et}(\wt Z_\infty,\O^+/p))\Rightarrow H^{n+m}_v(\wt Z_0,\O^+/p).\]
	Since $\wt Z_\infty=\varprojlim Z_{\infty,i}\to Z_{\infty,0}$ is a pro-finite-\'etale morphism of perfectoid spaces, we have
	\[ H^m_{\et}(\wt Z_\infty,\O^+/p)=\varinjlim_i H^m_{\et}(Z_{\infty,i},\O^+/p)=\varinjlim_i\varinjlim_j H^m_{\et}(Z_{j,i},\O^+/p)\]
	 by \Cref{c:tilde-limit-b0x}.
	We deduce by \cite[Proposition 1.2.5]{NeuSchWin} that we now have for any $n,m\geq 0$:
	\[ H^n_{\cts}(G,H^m_{\et}(\wt Z_\infty,\O^+/p))=\varinjlim_i\varinjlim_j H^n(G_j, H^m_{\et}(Z_{j,i},\O^+/p)).\]
	But these are the terms appearing in the usual \'etale Cartan--Leray sequence for $Z_{j,i}\to Z_{0,i}$:
	\[ H^n(G_j,H^m_{\et}(Z_{j,i},\O^+/p))\Rightarrow H^{n+m}_{\et}(Z_{0,i},\O^+/p).\]
	Thus the abutment of the first sequence is
	$=\varinjlim_iH^{n+m}_{\et}(Z_{0,i},\O^+/p)$,
	as we wanted to see.
	
	The case of $\bOx$ is similar, but instead using only the $5$-term exact sequence
	\[0\to H^1_{\cts}(G,\bOx(\wt Z_{\infty}))\to H^1_{v}(\wt Z_0,\bOx)\to H^1_{\et}(\wt Z_\infty,\bOx)^{G}\to  H^2_{\cts}(G,\bOx(\wt Z_{\infty}))\]
	which by the above arguments is the colimit over $i$ and $j$ of
	\[0\to H^1_{\cts}(G_j,\bOx(Z_{j,i}))\to H^1_{\et}(Z_{0,i},\bOx)\to H^1_{\et}( Z_{j,i},\bOx)^{G_j}\to  H^2_{\cts}(G_j,\bOx(Z_{j,i})).\]
	This finishes the proof of Step 2.
	\end{proof}
	\begin{step}
		Part 1 holds.
	\end{step}
	\begin{proof}
		This follows from Step 2 as the special case of $\wt X=X$.
	\end{proof}
	\begin{step}
		Part 2 holds for affinoid perfectoid $Y$.
	\end{step}
	\begin{proof}
		It follows from part 1 that the left hand side of the statement of part 2 is equal to $\varinjlim H^n_{\et}(X_i\times Y,\mathcal F)$. Thus this follows from Step 2.
	\end{proof}
	
	\begin{step}
		Part 2 holds for general spatial diamonds.
	\end{step}
	\begin{proof}Any spatial diamond admits a cover $\wt Y\to Y$ by an affinoid perfectoid space $\wt Y$ such that all finite products $\wt Y\times_Y\dots\times_Y\wt Y$ are affinoid perfectoid (e.g\ use \cite[Propositions~11.5, 11.14 and Lemma~7.19]{etale-cohomology-of-diamonds}). Comparing the \cH-to-sheaf spectral sequence
	\[ \cH^n(\wt Y\to Y,H^m_{v}(X_i\times -,\mathcal F))\Rightarrow H^{n+m}_{v}(X_i\times Y,\mathcal F)\]
	to the sequence for $X_i$ replaced by $\wt X$, we deduce the result from Step 4.
	\end{proof}
	This finishes the proof of \Cref{p:bOx-cohom-et-vs-v}.
\end{proof}

Towards \cref{l:describing-master-diagram-1}, we can use this to describe the cohomology sheaves:
	\begin{Lemma}\label{l:explicit-description-of-R^1pi_{tau*}F}
		Let $X,Y$ be locally spatial diamonds over $K$ with $X(K)\neq \emptyset$. Let $\tau$ be either the \'etale or the $v$-topology and let $F$ be a $\tau$-sheaf on $\LSD_K$ such that the pullback $F\to \pi_{\ast}F$ of sheaves on $Y_\tau$ along $\pi:X\times Y\to Y$ is an isomorphism. Then the Leray sequence
		\[ 0\to H^1_{\tau}(Y,\pi_{\ast}F)\to H^1_{\tau}(X\times Y,F)\to R^1\pi_{\tau\ast}F(Y)\to 1\]
		is a short exact sequence.
	\end{Lemma}
	\begin{proof}
		This is a standard argument that we learned from Gabber's simplification of \cite[Lemma~5]{Geisser-PicardScheme}:
		The full Leray 5-term exact sequence is of the form
		\[0\to H^1_{\tau}(Y,\pi_{\ast}F)\to H^1_{\tau}(X\times Y,F)\to R^1\pi_{\ast}F(Y)\to H^2_{\tau}(Y,\pi_{\ast}F)\to H^2_{\tau}(X\times Y,F).\]
		By assumption, the last map agrees with the pullback map
		\[\pi^{\ast}:H^2_{\tau}(Y,F)\to H^2_{\tau}(X\times Y,F).\]
		Any point $x\in X(K)$ now defines a splitting of $\pi^{\ast}$, showing that this map is injective. 
	\end{proof}

\begin{proof}[Proof of \cref{l:describing-master-diagram-1}.(A)-(B)]
	For part (A), we consider the a priori left exact sequence
	\[	1\to \pi^\diamondsuit_{\et\ast}\Oone\to {\pi^\diamondsuit_{\et\ast}\O^\times\to \pi^\diamondsuit_{\et\ast}\bOx} \to 1.\]
	By \cref{c:pi_astO^+} and
	\cref{l:H^1(bOx)-on-XxY-insensitive-to-sec-factor}, this gets identified with
	\[ 1\to \Oone\to \O^\times\to \bOx \to 1,\]
	which is short exact. Hence the boundary map $\pi^\diamondsuit_{\et\ast}\bOx\to R^1\pi^\diamondsuit_{\et\ast}\Oone$ vanishes. This shows the statement for the middle row. The other rows are completely analogous.
	
	 The first part of (B) was \cref{c:Lemma-B}. To finish the proof of (B) it remains to prove that the map
	\[ R^1\pi_{\et\ast}^\diamond\bOx\to R^1\pi_{v\ast}^{\diamond}\bOx\]
	is an isomorphism.	We may prove this locally on $\Perf_{K,\et}$, and may therefore assume that $X(K)\neq \emptyset$. By Lemma~\ref{l:H^1(bOx)-on-XxY-insensitive-to-sec-factor}, we can then apply Lemma~\ref{l:explicit-description-of-R^1pi_{tau*}F} to get an exact sequence
	\[ 1\to H^1_{v}(Y,\bOx)\to H^1_{v}(X\times Y,\bOx)\to R^1\pi^{\diamond}_{v\ast}\bOx(Y)\to 1.\]
	It also applies for the \'etale topology, so we also get a short exact sequence
	\[ 1\to H^1_{\et}(Y,\bOx)\to H^1_{\et}(X\times Y,\bOx)\to R^1\pi^{\diamond}_{\et\ast}\bOx(Y)\to 1.\]
	The first two terms of these sequences are isomorphic via the natural maps by Proposition~\ref{p:bOx-cohom-et-vs-v}.1. Thus the third terms are isomorphic.
\end{proof}

\section{Proof of Main Theorem}
At this point we have completed the proof of Lemma~\ref{l:describing-master-diagram-2} and of (A)--(B) of Lemma~\ref{l:describing-master-diagram-1}.

 We are left to prove Lemma~\ref{l:describing-master-diagram-1}.(C) and to explain how to deduce Proposition~\ref{p:relative-HT-for-1+m}, which is not completely formal from the diagram. Finally, we need to prove \cref{l:Pic_et-is-v-sheaf}.
 
 We can assume that $X$ is connected. Fix a base point $x\in X(\overline K)$. We can then define the universal pro-finite-\'etale cover 
 from \cite[\S3.4]{heuer-v_lb_rigid}: This is
the diamond $\wt X$ over $K$ defined as
\[\wt\pi:\wt X:=\varprojlim_{X'\to X}X'\to \Spd(K)\]
where the limit ranges over connected finite \'etale covers $(X',x')\to (X,x)$ with $x'\in X'(\overline K)$ a choice of lift of the base point $x$. This is a spatial diamond, and the canonical projection $\wt X\to X$
is a pro-finite-\'etale torsor under the \'etale fundamental group $\pi_1(X,x)$ of $X$.

We first note that we have an analogue of Corollary~\ref{c:pi_astO^+} in the inverse limit:
\begin{Lemma}\label{l:wtpi_astO^+}
	We have $\wt \pi_{\ast}\O=\O$ on $\Perf_{K,v}$, and similarly for $\Oone$, $\O^\times$, $\O^+$ and $\O^{+a}/p^k$.
\end{Lemma}
\begin{proof}
	We start with $\O^{+}/p^k$: For this we  deduce from
 \cref{p:bOx-cohom-et-vs-v}.2 and \cref{c:pi_astO^+}:
\[ \O^+/p^k(\wt X\times Y)\aeq \varinjlim \O^+/p^k(X'\times Y)\aeq \O^+/p^k(Y).\]The case of $\O$ follows by taking the limit over $k$ and inverting $p$. The cases of $\O^+$, $\O^\times$ and $\Oone$ follow as these are subsheaves.
\end{proof}
We are finally equipped to prove that the Hodge--Tate sequence for $\Oone$ is short exact:
\begin{proof}[Proof of Proposition~\ref{p:relative-HT-for-1+m}]
	We consider the morphism of logarithm long exact sequences
	\[
	\begin{tikzcd}[column sep = 0.3cm]
		{\pi^\diamondsuit_{\ast}\O} \arrow[r] &{R^1\pi^\diamondsuit_{\tau\ast}\mu_{p^\infty}} \arrow[r]           & {R^1\pi^\diamondsuit_{\tau\ast}\Oone} \arrow[r]           & {R^1\pi^\diamondsuit_{\tau\ast}\O} \arrow[r]           & {R^2\pi^\diamondsuit_{\tau\ast}\mu_{p^\infty}} \\
		{(\pi_\ast\O)^\diamondsuit} \arrow[u,"\sim"labelrotate] \arrow[r] &{(R^1\pi_{\et\ast}\mu_{p^\infty})^\diamondsuit} \arrow[u,"\sim"labelrotate] \arrow[r] & {(R^1\pi_{\et\ast}\Oone)^\diamondsuit} \arrow[u] \arrow[r] & {(R^1\pi_{\et\ast}\O)^\diamondsuit} \arrow[u] \arrow[r] & {(R^2\pi_{\et\ast}\mu_p^\infty)^\diamondsuit} \arrow[u,"\sim"labelrotate]
	\end{tikzcd}
	\]
	for $\tau$ one of $\et$ and $v$. For either topology, the first vertical arrow is an isomorphism by Corollary~\ref{c:pi_astO^+}. The second and fifth arrow are isomorphisms by Corollaries~\ref{c:base-change-for-etale-sheaves} and~\ref{c:R^npi_Fp-comparison}. 
	
	For the \'etale topology, also the fourth arrow is an isomorphism by Proposition~\ref{p:HTses-of-R^1pi-for-O}.1, and we conclude the first part by the 5-Lemma.
	
	For the $v$-topology, by splicing diagram \eqref{eq:master-dg-2} into short exact sequences, we can still deduce from Proposition~\ref{p:HTses-of-R^1pi-for-O}.2 that there is a left-exact sequence
	\[ 0\to (R^1\pi_{\et\ast}\Oone)^\diamondsuit\to R^1\pi^\diamondsuit_{v\ast}\Oone\xrightarrow{\HT\log}  H^0(X,\wt\Omega^1_X)\otimes_K\G_a.\]

	We are left to prove right-exactness. For this it suffices to prove that the map
	\[ \log:R^1\pi^\diamondsuit_{v\ast}\Oone\to R^1\pi^\diamondsuit_{v\ast}\O\]
	is surjective. To show this, we first assume that $K$ is algebraically closed, then we can argue like in \cite[\S3.5]{heuer-v_lb_rigid} (see the discussion surrounding \cite[diagram (10)]{heuer-v_lb_rigid}): For any affinoid perfectoid $Y$, consider the pro-finite-\'etale Galois cover
	\[\wt X\times Y\to X\times Y\]
	with group $G:=\pi_1(X,x)$. 
	By Lemma~\ref{l:wtpi_astO^+}, we have $H^0_v(\wt X\times Y,\O^+)=\O^+(Y)$.
	The Cartan--Leray sequence thus combines with the logarithm to a commutative diagram:
	\begin{equation}\label{eq:diag-HTlog-right-exact}
	\begin{tikzcd}[column sep = 0.2cm]
	\Hom_{\cts}(G,\Oone(Y)) \arrow[r] \arrow[d,"\log"] & H^1_v(X\times Y,\Oone) \arrow[d,"\log"] \\
	\Hom_{\cts}(G,\O(Y)) \arrow[r]        & H^1_v(X\times Y,\O).
	\end{tikzcd}
	\end{equation}
	We aim to see that the right vertical map  becomes surjective upon sheafication in $Y$. The left morphism becomes surjective since $\log:\Oone\to \O$ is and since the maximal torsionfree abelian pro-$p$-quotient of $G=\pi_1(X,x)$ is a finite free $\Z_p$-module \cite[Corollary~3.12]{heuer-v_lb_rigid}. It thus suffices to see that the bottom map is surjective: By \cref{p:affinoid-perfectoid-primitive-comparison}.2.(ii) we have  $H^1_v(X\times Y,\O)=H^1_v(X,\O)\otimes \O(Y)$. Since $\Hom_{\cts}(G,\O(Y))= \Hom_{\cts}(G,K)\otimes_K \O(Y)$, this map is $-\otimes_K\O(Y)$ applied to the same map in the case of $Y=\Spa(K)$,
	\[ \Hom_{\cts}(\pi_1(X,x),K)\to H^1_v(X,\O),\]
	which is an isomorphism since $H^1(\wt X,\O)=0$ by \cite[Proposition~4.9]{heuer-v_lb_rigid}.
	
	Returning to the case of general perfectoid $K$, consider the inverse system of finite sub-extensions $K\subseteq L\subseteq C$, then there is base-change morphism of logarithm exact sequences
	\[\begin{tikzcd}
		{H^1_v(X\times Y_L,\Oone)} \arrow[r,"\log"] \arrow[d] & {H^1_v(X\times Y_L,\O)} \arrow[r,"\partial_L"] \arrow[d] & {H^2_v(X\times Y_L,\mu_{p^\infty})} \arrow[d] \\
		{H^1_v(X\times Y_C,\Oone)} \arrow[r,"\log"] & {H^1_v(X\times Y_C,\O)} \arrow[r] & {H^2_v(X\times Y_C,\mu_{p^\infty})}
	\end{tikzcd}\]
	Let $x$ be an element in the middle of the top row, then by the algebraically closed case, the image of $x$ in the bottom row can be lifted along $\log$ after passing to an \'etale cover of $Y_C$. Using that $Y_{C,\etqcqs}=2\text{-}\varinjlim Y_{L,\etqcqs}$, we can assume that this \'etale cover comes via pullback from $Y_L$. Replacing $Y_L$ by this cover, we see that the image of $\partial_L(x)$ in the bottom row vanishes. But the rightmost vertical map becomes an isomorphism in the colimit over $L$ by \cite[Proposition~14.9]{etale-cohomology-of-diamonds}. Hence $\partial_L(x)$ vanishes for $L$ large enough, and we find the desired lift on the \'etale cover $Y_L\to Y$.
\end{proof}
Finally, we need to see that $(R^2\pi_{\et\ast}\Oone)^\diamond\to R^2\pi^\diamond_{v\ast}\Oone$ is injective:
\begin{proof}[Proof of Lemma~\ref{l:describing-master-diagram-1}.(C)]
	We argue as in the last proof, but in one degree higher: Let 
	\begin{align*}
	C_1&:=\coker(\log\colon R^1\pi_{\et\ast}\Oone\to R^1\pi_{\et\ast}\O)^\diamond,\\ C_2&:=\coker(\log\colon R^1\pi^\diamond_{v\ast}\Oone\to R^1\pi^\diamond_{v\ast}\O).
	\end{align*}
	By Propositions~\ref{p:relative-HT-for-1+m} and \ref{p:HTses-of-R^1pi-for-O}, these fit into a commutative diagram
	\[\begin{tikzcd}
		R^1\pi^\diamond_{v\ast}\Oone \arrow[r] & R^1\pi^\diamond_{v\ast}\O \arrow[r] & C_2 \\
		(R^1\pi_{\et\ast}\Oone)^\diamondsuit \arrow[u] \arrow[r] & (R^1\pi_{\et\ast}\O)^\diamondsuit \arrow[u] \arrow[r] & C_1, \arrow[u]
	\end{tikzcd}\]
	in which the first two vertical maps both have cokernel $\wtOm_X^1\otimes \G_a$. As the second vertical map is injective, this shows that the natural map $C_1\to C_2$ is an isomorphism.
	Continuing the outer diagram in \eqref{eq:master-dg-2} further to the right,  these terms fit into a long exact sequence
		\[
	\begin{tikzcd}[column sep = 0.4cm]
		0\arrow[r] &{C_2} \arrow[r] &{R^2\pi^\diamondsuit_{v\ast}\mu_{p^\infty}} \arrow[r]           & {R^2\pi^\diamondsuit_{v\ast}\Oone} \arrow[r]           & {R^2\pi^\diamondsuit_{v\ast}\O} \arrow[r]           & {R^3\pi^\diamondsuit_{v\ast}\mu_{p^\infty}} \\
	0\arrow[r] &{C_1} \arrow[u,"\sim"labelrotate] \arrow[r] &{(R^2\pi_{\et\ast}\mu_{p^\infty})^\diamondsuit} \arrow[u,"\sim"labelrotate] \arrow[r] & {(R^2\pi_{\et\ast}\Oone)^\diamondsuit} \arrow[u] \arrow[r] & {(R^2\pi_{\et\ast}\O)^\diamondsuit} \arrow[u,hook] \arrow[r] & {(R^3\pi_{\et\ast}\mu_{p^\infty})^\diamondsuit} \arrow[u,"\sim"labelrotate]
	\end{tikzcd}
	\]
	in which the first, second and last vertical arrows are isomorphisms by \cref{c:part-E-of-master-lemma-2}. The fourth arrow is injective by \cref{p:HTses-of-R^1pi-for-O-s4}. It follows that the middle arrow is injective.
\end{proof}

This finishes the proof of Lemmas~\ref{l:describing-master-diagram-1} and~\ref{l:describing-master-diagram-2}. As explained in \cref{s:outline}, this in turn completes the proof parts 1 and 2 of  Theorem~\ref{t:representability-of-v-Picard-functor}.

To get the partial splitting in part 3, we first introduce some notation: Let \[\mathcal A:=H^0(X,\wt\Omega^1_X)\otimes_K\G_a.\] We now use that by definition of $\alpha$, the exponential $\exp:p^\alpha\O^+\to \G_m$ defines a partial inverse to $\log$ on the subspace $1+p^\alpha \O^+\subseteq \G_m$. Moreover, for varying affinoid perfectoid $Y$, the natural maps
$H^1_{v}(X,\O^+)\otimes_{K^+} \O^+(Y)\to H^1_v(X\times Y,\O^+)$
induce a morphism
\[H^1_v(X,\O^+)\otimes p^\alpha\G_a^+\to R^1\pi^\diamondsuit_{v\ast}p^\alpha\O^+.\] Let $\mathcal A^+\subseteq \mathcal A$ be its image under  $R^1\pi^\diamondsuit_{v\ast}p^\alpha\O^+\to R^1\pi^\diamondsuit_{v\ast}\O\to \mathcal A$, then
these combine to a commutative diagram:
\[\begin{tikzcd}
	0 \arrow[r]& {\uPic^\diamondsuit_{X,\et}} \arrow[r] \arrow[r] & {\uPic^\diamondsuit_{X,v}} \arrow[r,"\HT\log"] & \mathcal A \\
	0 \arrow[r] & {H^1_{\et}(X,\O^+)\otimes p^\alpha\G_a^+} \arrow[r] \arrow[u] & {H^1_v(X,\O^+)\otimes p^\alpha\G_a^+} \arrow[r,"\HT^+"] \arrow[u,"\exp"'] & \mathcal A^+\arrow[u]
\end{tikzcd}\]
Since the image of $H^1_v(X,\O^+)\to H^1_v(X,\O)$ is an almost finite free $K^+$-module, we can find a splitting
 $s:H^0(X,\wtOm)\to H^1_v(X,\O)$ that induces a splitting of $\HT^+$, and thus of $\HT\log$.
 
The part of 2 about tangent spaces also follows from the bottom row of the diagram.

This finishes the proof of the Diamantine Picard Comparison \cref{t:representability-of-v-Picard-functor}.
\qed

\medskip

As usual, Theorem~\ref{t:representability-of-v-Picard-functor} in fact yields a precise description of the Picard group:

\begin{Corollary}\label{c:explicit-description-of-Pic}
	Let $Y$ be a perfectoid space over $K$ and assume that the rigid Picard functor $\uP_{X,\et}$ is represented by an adic space $G$. Then any $x\in X(K)$ defines an isomorphism
	\[ \Pic_{\et}(X\times Y)=\Pic_{\et}(Y)\times G(Y).\]
\end{Corollary}
\begin{proof}
	This follows from Theorem~\ref{t:representability-of-v-Picard-functor} and Lemma~\ref{l:explicit-description-of-R^1pi_{tau*}F} which applies by \cref{c:pi_astO^+}.
\end{proof}

\begin{proof}[Proof of \cref{l:Pic_et-is-v-sheaf}]
	\begin{enumerate}[leftmargin=*]
		\item 
		By Theorem~\ref{t:representability-of-v-Picard-functor}.1-2, the sheaf $(\uP_{X,\et})^\diamondsuit$ is the kernel of a morphism of $v$-sheaves on $\Perf_K$, hence it is itself a $v$-sheaf.
		\item Let us for simplicity write $\uP_{X,\et}^\diamond$ for the Picard functor defined on all of $\LSD_{K,\et}$, and similarly for the $v$-topology.
		We claim that for rigid $Y$, the natural sequence
		\[ 1\to \uP^\diamond_{X,\et}(Y)\to \uP^\diamond_{X,v}(Y)\xrightarrow{\HT\log}H^0(X,\wt\Omega^1_X)\otimes_K\O(Y)\]
		is still left-exact. 
		It then follows that $\uP^\diamond_{X,\et}$ is still the kernel of $\HT\log$ on rigid spaces. But $\HT\log$ is a morphism of $v$-sheaves, and thus its kernel is a $v$-sheaf.
		
		To see that the sequence is left-exact, we study the following commutative diagram:
		\[\begin{tikzcd}[column sep = 0.3cm]
			1 \arrow[r] & {H^1_{\et}(Y,\O^\times)} \arrow[r] \arrow[d] & {H^1_{\et}(X\times Y,\O^\times)} \arrow[r] \arrow[d] & {\uP^\diamond_{X,\et}(Y)} \arrow[r] \arrow[d] & 1 \\
			1 \arrow[r] & {H^1_{v}(Y,\O^\times)} \arrow[d, "\HT\log"]  \arrow[r]& {H^1_v(X\times Y,\O^\times)} \arrow[r] \arrow[d, "\HT\log"] & {\uP^\diamond_{X,v}(Y)} \arrow[r] \arrow[d, "\HT\log"] & 1 \\
			0 \arrow[r] & {H^0(Y,\wtOm^1_Y)} \arrow[r] & {H^0(X\times Y,,\wtOm^1_{X\times Y})}\arrow[r]  & {H^0(X,\wtOm^1_X)\otimes \O(Y)} \arrow[r] & 0
		\end{tikzcd}\]
		Here the first two columns are the left exact Hodge--Tate logarithm sequences \cite[Theorem~1.3]{heuer-v_lb_rigid} associated to the rigid spaces $Y$ and $X\times Y$, respectively. The first two rows are the exact sequences from Lemma~\ref{l:explicit-description-of-R^1pi_{tau*}F}. The bottom row is also exact, since
		\[\Omega^1(X\times Y)=\big(\Omega^1(Y)\otimes_K\O(X)\big)\oplus \big(\Omega^1(X)\otimes_K \O(Y)\big)\]
		and $\O(X)=K$. The leftmost $\HT\log$ in the diagram becomes surjective after \'etale sheafification in $Y$. It follows from a diagram chase that also the third column is exact.
		\item Clear from Theorem~\ref{t:representability-of-v-Picard-functor}.1.
		\item By part 2, if $Y'\to Y$ is a pro-\'etale perfectoid cover of a rigid space, we have descent for \'etale line bundles along $X\times Y'\to X\times Y$. One also has $v$-descent for maps into $G$.
		\item When $\uPic_{X,v}^\diamond$ is representable by a rigid group, then the short exact sequence \Cref{seq:Picv-ses} in \Cref{t:representability-of-v-Picard-functor}.2 expresses $\uPic_{X,\et}^\diamond$ as the kernel of a morphism of rigid groups, hence $\uPic_{X,\et}^\diamond$ is itself represented by a rigid group.
		
		To see the converse, let $\mathcal A^+$ be the bounded open subgroup of $\mathcal A= H^0(X,\wt\Omega^1_X)\otimes_K\G_a$ described in \Cref{t:representability-of-v-Picard-functor}.3 and consider for any $n\in \N$ the short exact sequence
		\[0\to \uP_{X,\et}\to \uP_{X,v}^{(n)}\to p^{-n}\mathcal A^+\to 0 \]
		defined by the fibre of \eqref{seq:Picv-ses} over the open subgroup $p^{-n}\mathcal A^+$.
		Then we have 
		\[ \uP_{X,v}=\cup_{n\in \N} \uP_{X,v}^{(n)},\]
		so it suffices to prove that each $\uP_{X,v}^{(n)}$ is represented by a rigid space.
		For any $n$ we have a morphism of short exact sequences of $v$-sheaves, exact in the \'etale topology
		\[ \begin{tikzcd}
			0 \arrow[r] & {\uP_{X,\et}} \arrow[d, "{[p^n]}"] \arrow[r] & {\uP_{X,v}^{(n)}} \arrow[d, "{[p^n]}"] \arrow[r] & p^{-n}\mathcal A^+ \arrow[d, "{\cdot p^n}","\sim"labelrotate] \arrow[r] & 0 \\
			0 \arrow[r] & {\uP_{X,\et}} \arrow[r] & {\uP_{X,v}^{(0)}} \arrow[r] & \mathcal A^+\arrow[r] & 0.
		\end{tikzcd}\]
		By Theorem~\ref{t:representability-of-v-Picard-functor}.3, the bottom sequence is split. In particular, the middle term is represented by a (smooth) rigid group variety if and only if the first term is. 
		
		By \cite[Lemma 15.6]{etale-cohomology-of-diamonds} any diamond that is \'etale over a rigid space comes from a rigid space, so it suffices to prove that the middle arrow is an \'etale morphism of diamonds.
		
		To see this, it suffices to prove that the left morphism is a finite \'etale morphism of diamonds: Indeed, if this is the case, then the short exact sequences express that $\uP_{X,v}^{(n)}\to \uP_{X,v}^{(0)}$ is \'etale-locally on $\mathcal A^+$ isomorphic to the finite \'etale morphism \[[p^n]\times p^n:\uP_{X,\et}\times p^{-n}\mathcal A^+\to \uP_{X,\et}\times \mathcal A^+\] defined as the product of the finite \'etale morphism $[p^n]:\uP_{X,\et}\to \uP_{X,\et}$ and the isomorphism $p^n:p^{-n}\mathcal A^+\to \mathcal A^+$. By v-descent of finite \'etale maps \cite[Proposition~10.11.(iii)]{etale-cohomology-of-diamonds}, it then follows that $\uP_{X,v}^{(n)}\to \uP_{X,v}^{(0)}$ is finite \'etale.
		
		To prove that $[p^n]:\uP_{X,\et}\to \uP_{X,\et}$ is finite \'etale, we may work $v$-locally and assume that $K$ is algebraically closed. Consider the sequence $1\to \mu_{p^n}\to \G_m \xrightarrow{p^n} \G_m\to 1$. By \Cref{c:pi_astO^+}, we have $\pi_{\et\ast}^\diamondsuit \G_m=\G_m$, hence this sequence stays exact after applying $\pi_{\et\ast}^\diamondsuit$. Using \cref{c:R^npi_Fp-comparison}, we conclude that the long exact sequence is therefore of the form
		\[ 1\to \underline{H^1_{\et}(X,\mu_{p^n})}\to \uP_{X,\et}\xrightarrow{[p^n]} \uP_{X,\et}\xrightarrow{\partial} \underline{H^2_{\et}(X,\mu_{p^n})}\to \dots.\]
		Since any morphism from a connected adic space to $\underline{H^2_{\et}(X,\mu_{p^n})}$ is constant, the zero locus of $\partial$ is given by a union of connected components of the rigid group $\uP_{X,\et}$. This shows that $[p^n]:\uP_{X,\et}\to \uP_{X,\et}$ is an $\underline{H^1_{\et}(X,\mu_{p^n})}$-torsor over its open and closed image, hence it is finite \'etale, as we wanted to see.\qedhere
	\end{enumerate}
\end{proof}

\subsection{Translation-invariant Picard functors}
If $X=A$ is a connected proper rigid group variety, i.e.\ an abeloid variety, then there is a variant of the Picard functor that is frequently used, for example by Bosch--L\"utkebohmert \cite[\S6]{BL-Degenerating-AV}: the translation-invariant Picard functor. We finish this section by noting that the Diamantine Picard Comparison Theorem easily implies a translation-invariant version. We will use this in \cite{heuer-isoclasses} to prove a uniformisation result for abeloids.
\begin{Definition}\label{d:translation-invariant-Pic}
	Let $A$ be a connected smooth proper rigid group. Denote by $\pi_1,\pi_2,m:A\times A\to A$ the two projection maps and the group operation, respectively. For any rigid or perfectoid space $Y$, we denote by $\Pic_{\et}^\tau(A\times Y)$ the kernel of the map
	\[ \pi_1^\ast+\pi_2^\ast-m^\ast:\Pic_{\et}(A\times Y)\to \Pic_{\et}(A\times A\times Y).\]
	The translation-invariant Picard functor $\uP_A^\tau$ of $A$ is defined as the kernel of the morphism
	\[ \pi_1^\ast+\pi_2^\ast-m^\ast:\uP_A\to \uP_{A\times A}.\]
	We analogously define the translation invariant diamantine Picard functor $\uP^{\diamond\tau}_{A,\et} \subseteq \uP^{\diamond}_{A,\et}$.
\end{Definition}
By duality theory of abeloids, developed by Bosch--L\"utkebohmert \cite[\S6]{BL-Degenerating-AV}, the functor  $\uP_A^\tau$ is represented by an abeloid variety $A^\vee$ that is called the dual abeloid.
 We deduce:
\begin{Corollary}
	We have $\uP^{\diamond\tau}_{A,\et}=(\uP^{\tau}_A)^\diamond$ and this functor is represented by $A^{\vee\diamond}$. In particular, for any perfectoid space $Y$ over $K$, we have
	$\Pic^\tau_{\et}(A\times Y)=\Pic(Y)\times A^\vee(Y)$.
\end{Corollary}
\begin{proof}
	The first statement follows from Theorem~\ref{t:representability-of-v-Picard-functor}.1 by exactness of $-^\diamond$. For the last part, we use that by \Cref{c:explicit-description-of-Pic}, we have
	$\Pic_{\et}(A\times Y)=\Pic(Y)\times \uP_{A}(Y)$, and similarly for $A\times A$. We get the desired statement by comparing kernels on both sides of the  maps in Definition~\ref{d:translation-invariant-Pic}.
\end{proof}

\appendix
\section{Lemmas on complexes of Banach algebras}
Let $R$ be any ring and let $\varpi\in R$ be any element. We recall that an $R$-module $M$ is said to have bounded $\varpi$-torsion if $M[\varpi^\infty]:=\cup_{k\in \N}M[\varpi^k]$ is equal to $M[\varpi^n]$ for some $n\in \N$.

\begin{Lemma}\label{l:bounded-torsion-4-term-sequence}
	If $A\to B\to C\to D$ is an exact sequence of $R$-modules in which $A$ and $D$ are killed by $\varpi^n$ for some $n\in \N$ and $B$ has bounded $\varpi$-torsion, then $C$ has bounded $\varpi$-torsion.
\end{Lemma}
\begin{proof}
	This is an elementary diagram chase: Choose $n$ large enough such that $B[\varpi^\infty]=B[\varpi^n]$. We claim that $C[\varpi^\infty]=C[\varpi^{2n}]$. Let $x\in C[\varpi^N]$ for some $N$. Then $\varpi^nx$ lifts to an element $y$ of $B$, and $\varpi^Ny$ goes to $\varpi^{N+n}x=0$ in $C$, hence lifts to $A$. Since $A$ is killed by $\varpi^n$, this implies $\varpi^{N+n}y=0$. Since $B[\varpi^\infty]=B[\varpi^n]$, this shows $\varpi^ny=0$, hence $\varpi^{2n}x=0$.
\end{proof}

\begin{Lemma}\label{l:two-out-of-three}
	Let $C_1^\bullet\to C_2^\bullet \to C_3^\bullet$ be a short exact sequence of complexes of $R$-modules. Suppose that for each $n\in \Z$, the module $H^n(C_3^\bullet)$ has bounded $\varpi$-torsion and $H^n(C_2^\bullet)$ is killed by $\varpi^k$ for some $k\in \N$. Then $H^n(C_1^\bullet)$ has bounded $\varpi$-torsion.
\end{Lemma}
\begin{proof}
	We apply \Cref{l:bounded-torsion-4-term-sequence} to the long exact sequence of cohomologies.
\end{proof}

Let $(K,K^+)$ be any non-archimedean field. We now specialise to the setting that $R=K^+$ and $\varpi\in K^+$ is a pseudo-uniformiser.

\begin{Lemma}\label{l:completion-of-Banach-cpx}
	Let $C^\bullet$ be a bounded complex of $\varpi$-torsionfree $\varpi$-adically complete $K^+$-modules. Suppose that $C^\bullet[\tfrac{1}{\varpi}]$ is a complex of $K$-Banach modules such that $H^n(C^\bullet[\tfrac{1}{\varpi}])$ is finite-dimensional for all $n\in \Z$. Then:
	\begin{enumerate}
		\item
	 $H^n(C^\bullet)$ has bounded $\varpi$-torsion.
	 \item  For any $\varpi$-torsionfree $K^+$-module $S$, we have
	 \[H^n(C^\bullet\hotimes_{K^+} S)= H^n(C^\bullet)\hotimes_{K^+} S\]
	 where $\hotimes$ denotes the $\varpi$-adically completed tensor product.
	\end{enumerate}
\end{Lemma}
\begin{proof}
	By \cite[§II.5 Lemma 1]{MumfordAV}, there exists a perfect complex $P^\bullet$ of $K$-vector spaces and a quasi-isomorphism $f:P^\bullet\to C^\bullet[\tfrac{1}{\varpi}]$. Choosing $K^+$-lattices in each $P^n$ and rescaling if necessary, we can find a perfect complex  $P^{+,\bullet}$ of $K^+$-modules such that $P^{+,\bullet}[\tfrac{1}{\varpi}]=P^\bullet$ and such that $f$ admits a $K^+$-model \[f^+:P^{+,\bullet}\to C^\bullet.\]
	Let $L$ be the mapping cone of $f^+$, then we have a short exact sequence of $K^+$-complexes
	\[0\to C^\bullet\to L\to P^{+,\bullet}[1]\to 0.\]
	Since $f$ is a quasi-isomorphism, $L[\tfrac{1}{\varpi}]$ is an exact complex of $K$-Banach modules, so we have $H^n(L)[\tfrac{1}{\varpi}]=0$. By a standard argument, it now follows from Banach's Open Mapping Theorem that $H^n(L)$ has bounded $\varpi$-torsion for all $n\in \N$. (We learnt this argument from the proof of  \cite[Proposition 6.10]{perfectoid-spaces}. See e.g.\ \cite[Lemma A.3.1]{heuer-thesis} for a proof of the statement.)
	
	The modules $H^n(P^{+,\bullet})$ are finitely presented $K^+$-modules, hence they also have bounded $\varpi$-torsion. Therefore, \Cref{l:two-out-of-three} applies and shows part 1.
	
	For part 2, we note that $S$ is a flat $K^+$-module. The result therefore follows from part 1 by \cite[Lemma~A.3.6]{heuer-thesis}, or alternatively \cite[Proposition~A.3.1]{heuer-relative-HT}.
\end{proof}

\end{document}